\documentclass{amsart}
\usepackage{amsmath,amssymb,amsthm,enumerate,multicol,graphicx,mathtools,enumitem,xcolor,hyperref}

\usepackage{appendix} 
\usepackage{listings} 

\usepackage{soul} 

\usepackage[margin=1in]{geometry}
\theoremstyle{plain}
\newtheorem{theorem}{Theorem}
\numberwithin{theorem}{section}

\newtheorem{proposition}[theorem]{Proposition}
\newtheorem{lemma}[theorem]{Lemma}

\theoremstyle{definition}

\newtheorem{example}[theorem]{Example}

\newcommand{\R}{\mathbb{R}}
\newcommand{\col}{\text{Col}^{R}_X}
\newcommand{\colA}{\text{Col}^{R}_{X_A}}
\newcommand{\Tcal}{\mathcal{T}}
\newcommand{\Ical}{\mathcal{I}}
\newcommand{\Scal}{\mathcal{S}}
\newcommand{\Mcal}{\mathcal{M}}

\newcommand{\Bcal}{\mathcal{B}}
\newcommand{\MV}{\mathcal{G}}

\title{Region colorings of surfaces in 4-space}
\author[]{Rom\'an Aranda, Noah Crawford, Andrew Maas, Nicole Marienau, Erica Pearce,\\ Renzo Sarreal, Savannah Schutte, Ransom Sterns, and Eric Woods}

\date{}

\begin{document}

\begin{abstract}
Niebrzydowski introduced a theory of region colorings for surface links. In this paper, we translate the coloring invariant to the context of triplane diagrams and movies of knots. We provide inequalities between the number of region colorings and topological quantities of $F$, such as the number of saddles in a movie and the bridge index of a triplane diagram of $F$. As an application, we show that Yoshikawa's 2-knots $9_1$ and $10_2$ are non-invertible; that is $F\not=-F$.
\end{abstract}

\maketitle

\section{Introduction}
One of the first ideas we see in a knot theory course is that \emph{coloring} the strands in a knot diagram is a useful technique to distinguish knots apart. Instead of coloring arcs, one would want to define a knot invariant by coloring the planar complementary regions in a knot diagram. Niebrzydowski introduced an algebraic object that can be used for such a purpose~\cite{Niebrzydowski_tribracket_1}. \emph{Knot-theoretic ternary quasigroups}, also called \emph{tribrackets}, are sets equipped with a \emph{triple product} that resembles the Reidemester moves for knots; see Figure~\ref{fig:R3}. Counting tribracket colorings of a knot diagram yields a knot invariant, denoted by $\col(K)$, that mimics the theory of quandles used to color the strands of knots. In recent years, tribrackets have been extensively studied and generalized to other diagrams of 1-dimensional knots, and robust polynomial and homological invariants have been introduced~\cite{Nelson_Polynomial, Trivalent_tribrackets, Niebrzydowski_Flocks, Tribracket_Modules, Virtual_Tribrackets, Local_Biquandles, Tribracket_Spatial, Niebrzydowski_tribracket_2}. 

\begin{figure}[h]
\centering
\includegraphics[width=.7\textwidth]{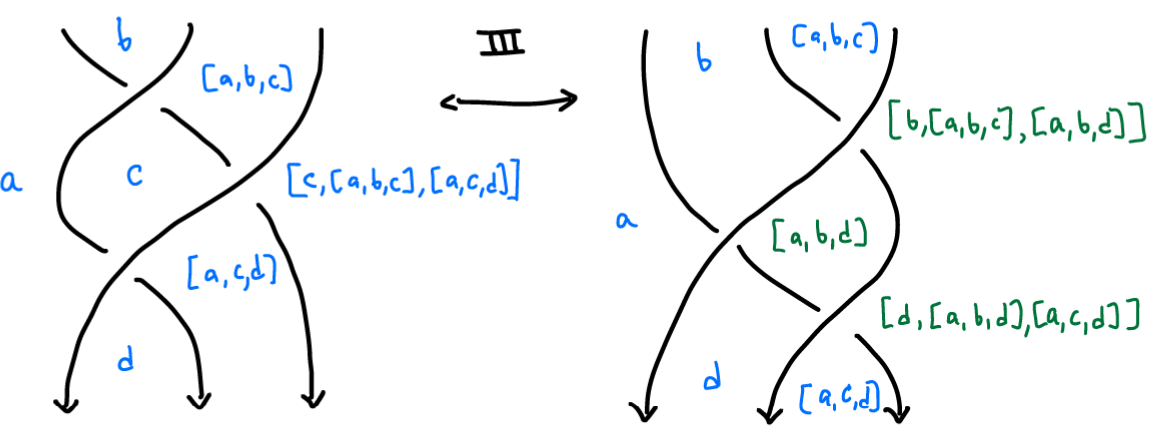}
\caption{Reidemeister III move and the tribracket equations.}
\label{fig:R3}
\end{figure}

One dimension up, the knotted objects are surfaces. Niebrzydowski introduced a region coloring invariant for surfaces in 4-space, denoted by $\col(F)$, using tribrackets and broken surface theory~\cite{Niebrzydowski_knotted_surfaces}. There are more ways to represent surfaces in 4-space, such as movies of knots, marked vertex diagrams, braid charts, and triplane and multiplane diagrams~\cite{book_surfaces_in_4space, MZ_bridge_trisections_S4, AE24_Multisections}. In this paper, we explore Niebrzydowski's tribracket coloring invariant for surfaces represented in some of these ways. We present alternative proofs of the isotopy invariance in each setup. We study relations between $\col(F)$ and topological quantities of $F$, such as the number of saddles in a movie (Prop~\ref{prop:movie_inequality_saddle}) and the bridge index of $F$ (Prop~\ref{prop:triplane_inequality_saddle}). We also provide formulae for $\col(F)$ in two situations: when $X$ is an abelian tribracket and $F$ is a spun knot. We refer to Section~\ref{sec:computations} for definitions of the concepts below. 

\begin{theorem}\label{thm:abelian_trib}
Let $A$ be a finite abelian group with associated Dehn tribracket denoted by $X_A$. For any orientable surface-link $F\subset S^4$, the number of $X_A$-colorings is equal to 
\[ 
\text{Col}_{X_A}^R(F)=|A|^{|F|+1}.
\]
\end{theorem}

\begin{theorem}\label{thm:spun_knots}
Let $K$ be an oriented knot in $S^3$. There is a bijection between the number of tribracket colorings of $K$ and that of the spun of $K$, $\Scal(K)$. In particular, if $X$ is a finite tribracket, then 
\[ 
\col(K) = \col(\Scal(K)).
\]
\end{theorem}

The final application of this work concerns non-invertible surface links, those with $ F \ neq -F$. In \cite{Niebrzydowski_knotted_surfaces}, Niebrzydowski used 2-cocycles of tribrackets to show that the 2-twist spin of the link $L6a5$ (in Thistlethwaite's notation) is non-invertible. In Section~\ref{sec:computations}, we use region colorings to do the same for two 2-knots from Yoshikawa's table~\cite{Yoshi94_enumeration}. 

\begin{theorem}\label{thm:noninvertible}
The 2-knots $9_1$ and $10_3$ are non-invertible.
\end{theorem}

\subsection{Outline of the paper.} Section~\ref{sec:prelims} reviews the theory of tribracket colorings for knots and surfaces in dimensions three and four, respectively. Sections~\ref{sec:movies} and \ref{sec:triplane_diagrams} translate $\col(F)$ for movies and triplane diagrams, respectively. Each section aims to provide a brief introduction to each diagrammatic theory. Section~\ref{sec:computations} contains proofs of Theorems~\ref{thm:abelian_trib} and \ref{thm:spun_knots} and a list of coloring equations necessary to compute $\col(F)$ for the knotted 2-spheres in Yoshikawa's surface-link table~\cite{Yoshi94_enumeration}. Section~\ref{sec:questions} lists questions and directions for future work. 

\subsection{Acknowledgments.} This project is the outcome of the January 2025 class MATH 398 - Research Experience in Mathematics at the University of Nebraska-Lincoln. RA thanks Alex Zupan for his advice during the course planning process.

\section{Niebrzydowski tribrackes}\label{sec:prelims}

A \emph{knot-theoretic ternary quasigroup (KTQ)} or \emph{Niebrzydowski tribracket} is a set $X$ equiped with a function $[-,-,-]:X\times X\times X \mapsto X$ satisfying the following conditions for any subset $\{a,b,c,d\}\subset X$, 
\begin{enumerate}
\item Any three elements uniquely determine the fourth such that $[a,b,c]=d$,
\item 
$\left[b,[a,b,c],[a,b,d]\right]=\left[c,[a,b,c],[a,c,d]\right]=\left[d,[a,b,d],[a,c,d]\right]$.

\end{enumerate}
The tribracket definition can be motivated by physics and Reidemeister moves of knots~\cite{Niebrzydowski_tribracket_1}. For instance, equations in (2) above can be explained by the third Reidemeister move in Figure~\ref{fig:R3}. An example of a tribracket structure is the \emph{Dehn tribracket} of a group $(G,\cdot)$ defined by $[a,b,c]=b\cdot a ^{-1}\cdot c$. The Dehn tribracket resembles relations in the Dehn presentation of the knot group. For a finite set $X$, we can describe its tribracket structure using a 3-tensor (array of $n\times n$ matrices). Below, we write an example from~\cite[Ex 10]{Tribracket_Modules} of a tribracket structure on the set $\{1,2,3,4\}$.
\[
\begin{bmatrix}
\begin{bmatrix}
4 & 3 & 2 & 1\\
2 & 4 & 1 & 3\\
3 & 1 & 4 & 2\\
1 & 2 & 3 & 4
\end{bmatrix}
& 
\begin{bmatrix}
3 & 1 & 4 & 2\\
4 & 3 & 2 & 1\\
1 & 2 & 3 & 4\\
2 & 4 & 1 & 3
\end{bmatrix}
&
\begin{bmatrix}
2 & 4 & 1 & 3\\
1 & 2 & 3 & 4\\
4 & 3 & 2 & 1\\
3 & 1 & 4 & 2
\end{bmatrix}
&
\begin{bmatrix}
1 & 2 & 3 & 4\\
3 & 1 & 4 & 2\\
2 & 4 & 1 & 3\\
4 & 3 & 2 & 1
\end{bmatrix}
\end{bmatrix}
\]

\subsection{Region colorings of knots in 3-space}
 
A link is an embedding of a disjoint union of circles into $\R^3$. A link with one connected component is called a knot. We are interested in knots and links up to isotopy in $\R^3$. A link projection (or link diagram) is a four-valent planar graph with crossing information at each vertex. Link diagrams determine embeddings of links in 3-space by embedding the associated graph in a 2-dimensional plane and pushing the edges near the vertices according to the crossing data. It is a classical result that any two link diagrams representing isotopic links differ by a sequence of Reidemeister moves~\cite{Reidemeister_moves}; see Figure~\ref{fig:Reid_moves}. 
\begin{figure}[h]
\centering
\includegraphics[width=.3\textwidth]{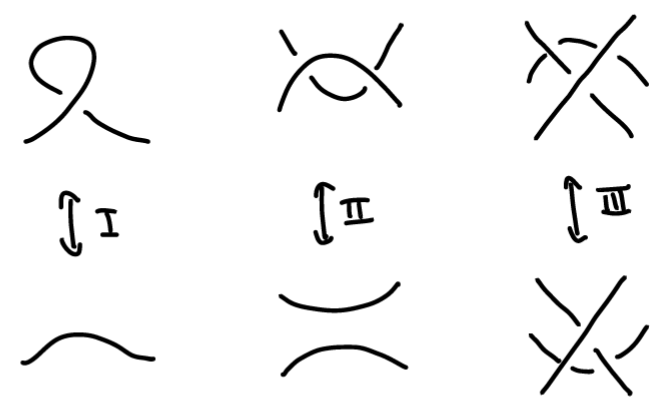}
\caption{Reidemeister moves.}
\label{fig:Reid_moves}
\end{figure}

Let $X$ be a tribracket, and let $D$ be an oriented link diagram. A \emph{region coloring} of $D$ is an assignment of elements of $X$ to the complementary regions of $D$; we will refer to the labels of each region as \emph{color}. A \emph{tribracket coloring} is a region coloring of $D$ so that the colors near each crossing satisfy the equations in Figure~\ref{fig:tribracket_equations}. Niebrzydowski introduced such coloring rules and showed that they induce link invariants. If $X$ is a finite tribracket, $\col(D)$ denotes the number of tribracket colorings of a diagram $D$.

\begin{lemma}[Lem 3.10 of \cite{Niebrzydowski_knotted_surfaces}]\label{lem:col_IS_knot_diagram}
    Reidemeister moves induce bijections between the sets of tribracket colorings of a link diagram. In particular, $\col(D)=\col(D')$ for link diagrams representing isotopic oriented links. 
\end{lemma}

\begin{figure}[h]
\centering
\includegraphics[width=.35\textwidth]{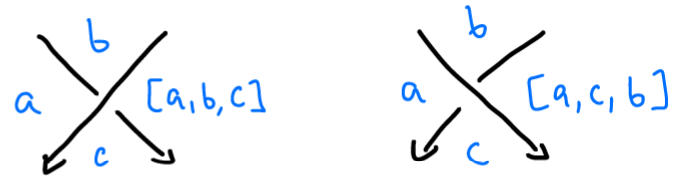}
\caption{Tribracket equations at a crossing.}
\label{fig:tribracket_equations}
\end{figure}

We end this subsection with three properties of tribracket colorings of link diagrams. An \emph{unlink} is a link isotopic to the disjoint union of embedded loops in the $xy$-plane. 

\begin{lemma}\label{lem:tribrackets_unlink}
    Let $X$ be a finite tribracket. If $L$ is an unlink, then $\col(L)=|X|^{|L|+1}$. 
\end{lemma}
\begin{proof}
    A crossingless diagram of $L$ has $|L|+1$ independent regions, so it has $|X|^{|L|+1}$ distinct tribracket colorings. The result then follows from Lemma~\ref{lem:col_IS_knot_diagram}.
\end{proof}

It is worth noting that tribracket colorings can be defined for knotted strands, such as tangles and braids. The following lemma will be used later in our paper. 
\begin{lemma}\label{lem:region_cols_braids}
    Let $\beta$ be a $n$-stranded braid. Suppose that the strands of $\beta$ are oriented in some way, not necessarily all pointing in the same direction. Then, the region coloring near the bottom determines each tribracket coloring of $\beta$. 
\end{lemma}
\begin{proof}
Condition (1) in the definition of tribracket implies that the colors of three regions near a crossing uniquely determine the coloring of the fourth crossing. Thus, a choice of region colorings near the bottom of a braid can be extended uniquely to a region coloring of the braid, one crossing at a time. 
\end{proof}

A \emph{tangle diagram} is a planar 4-valent graph with crossing information embedded in a disk $D^2$ with boundary a finite set of points in the boundary of $D^2$. Tangle diagrams describe knotted strands embedded in a 3-dimensional ball. In Section~\ref{sec:triplane_diagrams}, we will work with tangle diagrams up to isotopies that fix the endpoints of the tangle. For example, Reidemeister moves away from the endpoints of the tangle preserve the isotopy class of a tangle. All of our $b$-stranded tangle diagrams will have the same set of endpoints numbered $1,\dots, 2b$. A tangle diagram is \emph{trivial} if it can be isotoped while fixing its endpoints so that each strand of the new tangle contains a single maximum point with respect to the natural height function on the diagram. Figure~\ref{fig:triplane_to_marked_2} includes examples of three 4-stranded trivial tangles. For a trivial tangle $T$, the symbol $\overline{T}$ denotes the \emph{mirror image} of $T$, which is obtained by reflecting $T$ over the boundary of the disk. The mirror image of a $b$-stranded trivial tangle has the property that $T\cup \overline{T}$ is a link diagram for a $b$-component unlink. 

\begin{lemma}\label{lem:colors_of_trivial_tangle}
The tribracket colorings of a trivial tangle are determined by the colors of the regions near its boundary. 
\end{lemma}
\begin{proof}
Let $T$ be a $b$-stranded trivial tangle diagram. By definition of a trivial tangle, we can perform interior Reidemeister moves so that the new diagram is formed by a $2b$-stranded braid in the bottom and with consecutive pairs of strings at the top being connected with $\cap$-shaped arcs. By Lemma~\ref{lem:region_cols_braids}, tribracket colorings of these new diagrams are determined by the colors near the boundary of the tangle. As Interior Reidemeister moves are supported away from the boundary of the tangle, they induce bijective correspondences between the tribracket colorings, and our claim holds.
\end{proof}

\subsection{Region colorings of surfaces in 4-space}

A surface-link is an embedding of a disjoint union of abstract surfaces into $\R^4$. In this paper, we will study oriented surface-links up to ambient isotopy in $\R^4$. One way to describe surface-links is to project the embedding onto some hyperplane $\R^3\times \{0\}\subset \R^4$ just as we do with knots in 3-space. A \emph{broken surface diagram} is an immersed surface in $\R^3$ with generic double points, triple points, and branch points described locally in Figure~\ref{fig:broken_picture}. Broken diagrams come equipped with \emph{crossing information}: broken arcs correspond to pieces of the surface-link that have been pushed towards $\R^3\times (-\varepsilon, 0)$. There is a list of seven local modifications of broken surface diagrams, called \emph{Roseman moves}, which are enough to relate any two broken diagrams describing the same surface in 4-space~\cite{Roseman_moves}. We refer the reader to~\cite{book_surfaces_in_4space} for more details on broken surface diagrams. 

\begin{figure}[h]
\centering
\includegraphics[width=.6\textwidth]{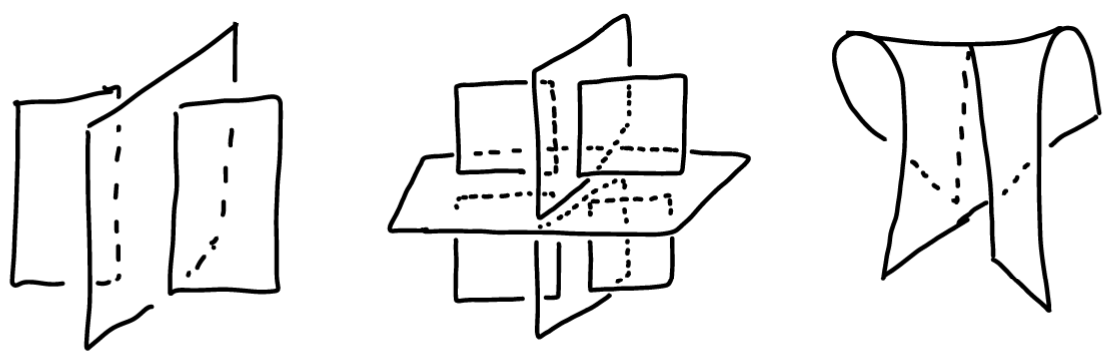}
\caption{Local models for double points, triple points, and branch points in a broken surface diagram.}
\label{fig:broken_picture}
\end{figure}

We will work with orientable surfaces in 4-space. In this case, an orientation of a broken surface diagram $\Bcal$ is a choice of a normal vector to each sheet of $\Bcal$ that matches when we pass through a double point as in Figure~\ref{fig:broken_rules}. Let $X$ be a tribracket and $\Bcal$ be an oriented broken surface diagram. A \emph{region coloring} of $\Bcal$ is an assignment of elements of $X$ to the 3-dimensional complementary regions $\Bcal$. A \emph{tribracket coloring} is a region coloring so that the colors near double points satisfy the equation in Figure~\ref{fig:broken_rules}. Nelson, Oshiro, and Oyamaguchi related Niebrzydowski tribracket colorings of broken diagrams with a different algebraic structure called \emph{local biquandles} \cite{Nelson_local_biquandles}. Niebrzydowski introduced such colorings and showed that the number of region tribracket colorings in a broken surface diagram, denoted by $\col(\Bcal)$, is a 4-dimensional isotopy invariant.

\begin{theorem}[Thm 3.18 of \cite{Niebrzydowski_knotted_surfaces}]\label{lem:broken_invariant}
    Roseman moves induce bijections between the sets of tribracket colorings of a broken surface diagram. In particular, $\col(\Bcal)=\col(\Bcal')$ for broken diagrams representing isotopic oriented surface-links. 
\end{theorem}

\begin{figure}[h]
\centering
\includegraphics[width=.26\textwidth]{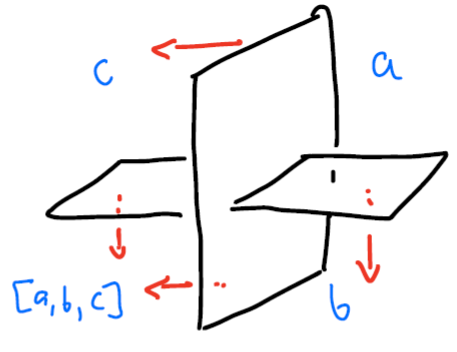}
\caption{Region colorings near double points.}
\label{fig:broken_rules}
\end{figure}

\begin{lemma}\label{prop:colorings_unknotted_surface}
Let $F$ be an oriented surface-link in $S^4$. If $F$ is unknotted, then 
\[ 
\col(F)=|A|^{|F|+1}, 
\]
where $|F|$ denotes the number of connected components of $F$. 
\end{lemma}
\begin{proof}
Orientable unknotted surfaces admit broken diagrams with no double or triple points. The result follows as any embedding of an $m$-component closed surface in $\R^3$ separates the space into $m+1$ distinct regions. 
\end{proof}

\section{Movies}\label{sec:movies}
Recall that a link diagram is a four-valent planar graph with crossing information at each vertex. A \emph{movie} is a one-parameter family of classical link diagrams $\{D_t:t\in [0,1]\}$ such that nearby diagrams differ by topological equivalences of the underlying 4-valent graphs or \emph{elementary string interactions}: a Reidemeister move, a birth, a death, or a saddle. See Figures~\ref{fig:Reid_moves} and~\ref{fig:string_interactions} for local descriptions of these interactions. Each link diagram $D_t$ will be referred to as a \emph{frame}. An \emph{orientation} of a movie $\Mcal$ is a consistent choice of orientations of all frames in $\Mcal$. Oriented movies induce embeddings of oriented surface-links into $\R^4$ as each frame represents an oriented link $L_t$ in $\R^3\times \{t\}\subset \R^3\times \R$. For more information on movies, see~\cite{book_surfaces_in_4space}. In 1993, Carter and Saito produced a set of fifteen local modifications of a movie that preserved the isotopy class of the surface-link. They showed that these \emph{movie moves} can be used to connect two movies representing isotopic surfaces \cite{movie_moves}; see Figure~\ref{fig:movie_moves}.

\begin{figure}[h]
\centering
\includegraphics[width=.4\textwidth]{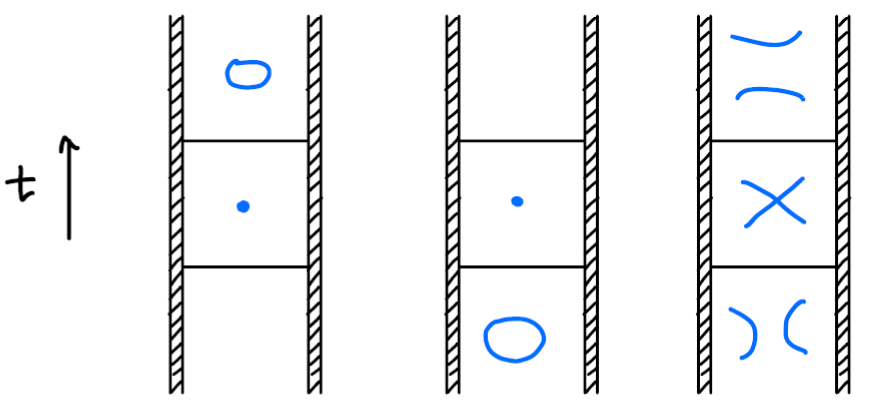}
\caption{Elementary string interactions.}
\label{fig:string_interactions}
\end{figure}

A tribracket coloring of an oriented movie $\Mcal=\{D_t\}_t$, is a consistent choice of tribracket colorings for each diagram $D_t$. Consistent refers to the condition that regions on the same side of a strand in a diagram $D_t$ have the same label while increasing the variable $t$. For consistent colorings, the labels in the four regions near a saddle point are related as in Figure~\ref{fig:string_interactions}. The following result is a direct consequence of Lemma~\ref{lem:col_IS_knot_diagram}. 
\begin{lemma}\label{rem:movies_of_Reid}
    The number of tribracket colorings of a movie consisting of Reidemeister moves only is equal to the number of tribracket colorings of any intermediate link diagram. 
\end{lemma}

For a tribracket $X$ and an oriented movie $\Mcal$, $\col(\Mcal)$ is defined to be the number of tribracket colorings of $\Mcal$. The following proposition implies that $\col(\Mcal)$ is an invariant of surface-links.

\begin{proposition}\label{prop:invariance_movie_moves}
Let $\mathcal{M}_1$ and $\mathcal{M}_2$ be two movies of surface-links. If they represent isotopic oriented surface-links, then $\col(\mathcal{M}_1)=\col(\mathcal{M}_2)$.
\end{proposition}
\begin{proof}
It is sufficient to verify that the number of tribracket colorings of a movie remains unchanged under Carter-Saito movie moves. We will refer to the moves in Figure~\ref{fig:movie_moves}. 
Lemma~\ref{rem:movies_of_Reid} implies that a pair of movies of Reidemeister moves with some frames in common must have the same number of tribracket colorings. Thus, the following movie moves must preserve the number of tribracket colorings: 1, 2, 3, 4, 5, 6, 7,  9, 10, and 11. 
In what follows, we treat the remaining moves separately: 8, 12, 13, 14, and 15. 
\begin{figure}[h]
\centering
\includegraphics[width=.7\textwidth]{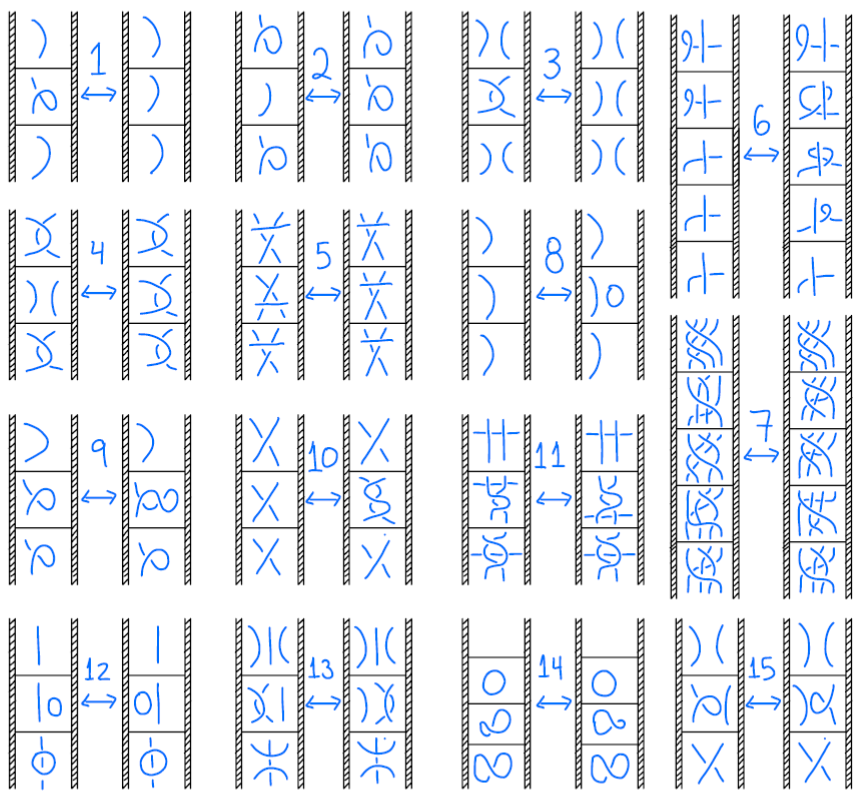}
\caption{The fifteen Carter-Saito movie moves.}
\label{fig:movie_moves}
\end{figure}

Move 8: The movies involved in this move correspond to one saddle followed by one birth. Locally, the coloring in the left movie is determined by the coloring in the first frame. The same holds for the second movie, as the saddle between the middle and last frames forces the color of the small disk to be dependent on the colors of the last frame, which, in turn, agrees with the colors of the first frame. Thus, move 8 preserves the tribracket coloring. 

Moves 12, 13, and 15: The movies involved in these moves correspond to one saddle followed by one Reidemeister move. For either movie, the fact that the middle frame connects with the top frame by a saddle forces the coloring in the middle frame. By Lemma~\ref{rem:movies_of_Reid}, the coloring between the middle and bottom frame is determined as they differ by Reidemeister moves. Thus, the coloring of either movie is determined by its top frame. Hence, this move does not change the number of colorings. 

Move 14: The movies involved in this move are composed of one birth and one Reidemeister move. By Lemma~\ref{rem:movies_of_Reid}, the colorings in either movie are determined by the colorings in the second to top frames of each movie, which are identical. Hence, the number of tribracket colorings is invariant under move 14. 
\end{proof}

\begin{figure}[h]
\centering
\includegraphics[width=1\textwidth]{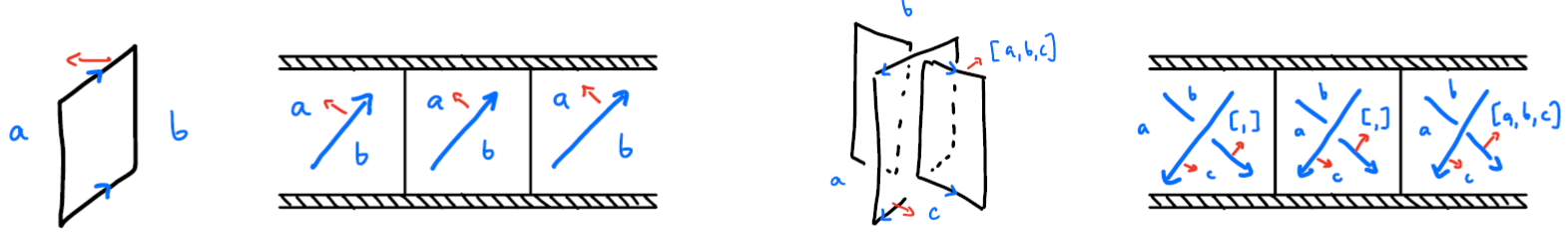}
\caption{How a movie describes a broken surface diagram}
\label{fig:movie_to_broken}
\end{figure}
Let $\Mcal$ be a movie representing an oriented surface-link $F\subset S^4$. By stacking the 2-dimensional frames of $\Mcal$ we obtain a broken surface diagram for $F$ denoted by $\Bcal_\Mcal$. When the movie $\Mcal$ is oriented, this induces an orientation on the broken diagram as in Figure~\ref{fig:movie_to_broken}. See~\cite{movie_moves} for details. The following proposition shows that the number of tribracket colorings of a movie equals $\col(F)$.

\begin{lemma}\label{lem:movies_give_broken}
Let $X$ be a tribracket. There is a bijection between the number of region $X$-colorings of $\Mcal$ and that of $\Bcal_\Mcal$. In particular, if $X$ is finite, 
\[ \col(\Mcal)=\col(\Bcal_\Mcal). \]
\end{lemma}

\begin{proof}
    Let $\Mcal$ be a movie and let $\Bcal_\Mcal$ be the associated broken diagram. As shown in Figure~\ref{fig:movie_to_broken}, the tribracket coloring rule for each frame of $\Mcal$ induces a region coloring of $\Bcal_\Mcal$ that satisfies the tribracket equations for the crossings of $\Bcal_\Mcal$. This association is bijective since different region colorings of $\Mcal$ produce distinct region colorings of $\Bcal_\Mcal$, and region colorings of $\Bcal_\Mcal$ can be ``sliced'' to produce region colorings of $\Mcal$. 
\end{proof}

\begin{lemma}\label{prop:movie_cross}
    Let $\Mcal$ be a movie representing a surface-link $F\subset S^4$ and let $L\subset S^3$ be a link that is a cross-section of $\Mcal$. If $L$ is above all the births in $\Mcal$, then $$\col(\Mcal)\leq \col(L).$$ 
\end{lemma}

\begin{proof}
    The second part of $\Mcal$, from the frame containing $L$ to the movie's end, comprises a sequence of Reidemeister moves, saddles, and deaths. Notice that a region coloring of a link can be uniquely extended to a region coloring of the link after a Reidemeister move, a saddle, or a death. Thus, a region coloring of $L$ determines the region coloring of all subsequent frames. By running the movie backward, we conclude that the coloring of $L$ determines the region colorings of all the frames before it. Hence, $\col(\Mcal)\leq \col(L)$. Notice that this inequality may be strict as some tribracket colorings of $L$ may extend to region colorings that violate the tribracket equations in some frames. 
\end{proof}

We count the number of births, saddles, and deaths in a movie $\Mcal$ and denote them by $m$, $s$, and $M$, respectively. These numbers are related through the Euler characteristic of $F$ through the equation $\chi(F)=m-s+M$. The quantities $m$ and $M$ satisfy that 
\[ 
\text{rank}\left(\pi_1(S^4-F)\right) \leq \min(m,M).
\]
We provide upper bounds to such quantities in terms of region colorings. Denote by $m(F)$, similarly $s(F)$ and $M(F)$, the smallest number of births in a movie among all movies representing $F$. 

\begin{proposition}\label{prop:movie_inequality_saddle}
    Let $F\subset S^4$ be an oriented surface. The following inequalities hold for a finite tribracket $X$. 
    \begin{enumerate}[label=(\roman*)]
    \item $-1+\log_{|X|}\left(\col(F)\right)\leq m(F)$, and 
    \item $ -2-\chi(F) + 2\log_{|X|}\left(\col(F)\right)\leq s(F)$.
    \end{enumerate}
\end{proposition}
\begin{proof}
    Let $\Mcal$ be a movie representing $F$. It is a known fact that $\Mcal$ can be modified without changing the surface $F$ nor the quantities $m$, $s$, nor $M$ so that all the births occur before all the saddles, which occur before all the deaths. In particular, the $m$-component unlink is a cross-section of $\Mcal$ above all births. Lemmas \ref{lem:movies_give_broken} and \ref{prop:movie_cross} imply that $\col(F)$ is bounded above by the number of tribracket colorings of the $m$-component unlink which, by Lemma~\ref{lem:tribrackets_unlink}, is equal to $|X|^{m+1}$. Thus, $\col(\Mcal)\leq |X|^{m+1}$ and so (i) holds. 
    To show (ii), notice that every movie played backward interchanges the births and deaths without changing $F$. Thus, $m(F)=M(F)$ and so 
    \[-2+2\log_{|X|}\left(\col(F)\right)\leq m(F)+M(F)=\chi(F)+s.\] 
\end{proof}

\subsection{Marked vertex diagrams}
The relevant information about movies of surface-links can be compressed into a graph that we now discuss. A \emph{marked vertex diagram} is a planar graph $\MV$ with two types of four-valent vertices called \emph{crossings} and \emph{marked vertices}. Crossings are the same as in a classical link diagram, and marked vertices are decorated with a small rectangle as in Figure~\ref{fig:smoothings}. We can remove a marked vertex of $\MV$ in one of two particular ways called \emph{A-smoothing} and \emph{B-smoothing}. Denote by $\MV_A$ and $\MV_B$ the link diagrams obtained by replacing all of the marked vertices with all-A-smoothing and all-B-smoothing, respectively. A marked vertex diagram $\MV$ is called \emph{simple} if $\MV_A$ and $\MV_B$ represent unlinks in 3-space. Simple marked vertex diagrams were studied by Lomonaco~\cite{lom81} and Yoshikawa~\cite{Yoshi94_enumeration}; for a complete survey of these diagrams see~\cite{book_surfaces_in_4space}.
\begin{figure}[h]
\centering
\includegraphics[width=.4\textwidth]{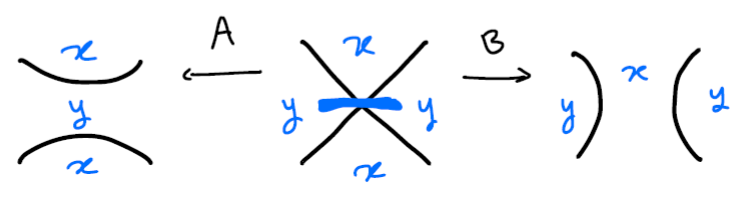}
\caption{Two distinct smoothings of a marked vertex. Notice that they correspond with the top and bottom frames of a saddle in Figure~\ref{fig:string_interactions}.}
\label{fig:smoothings}
\end{figure}

An \emph{orientation} of a marked vertex diagram $\MV$ is a choice of orientations for the edges of $\MV$ so that the oriented graphs obtained by smoothing all the marked vertices (in any way) are oriented as classical link diagrams. Notice that, near a marked vertex, the orientation of the edges must be the same for the NE and SW edges, which have the opposite orientations as the NW and SE edges. For an oriented marked vertex diagram $\MV$, a region coloring of $\MV$ is called a \emph{tribracket coloring} if it induces tribracket colorings for $\MV_A$ and $\MV_B$. Notice that the labels of the regions near a marked vertex are related as in Figure~\ref{fig:smoothings}. For a finite tribracket $X$ and an oriented marked vertex diagram, $\col(\MV)$ denotes the number of tribracket colorings. 

We now explain how marked vertex diagrams describe surface-links in 4-space. Let $\MV$ be a simple marked vertex diagram. Using the movie of a saddle from Figure~\ref{fig:string_interactions}, one can create a movie that starts in $\MV_A$ and ends in $\MV_B$. As both $\MV_A$ and $\MV_B$ are diagrams for the unlink, there are sequences of Reidemeister moves taking them to crossingless diagrams, which can then be completed with movies of births and deaths to a movie for a closed surface in 4-space. Notice that there are many sequences of Reidemeister moves that turn an unlink diagram into a crossingless diagram. That said, it is a fact that the isotopy class of surface-link built in this way is the same regardless of the sequence of Reidemeister moves we choose; see~\cite{book_surfaces_in_4space} for more details, or one can prove it using the Main Theorem of~\cite{livingston_unlink}. 

\begin{proposition}
    Let $\MV_1$ and $\MV_2$ be two oriented, simple marked vertex diagrams. If $\MV_1$ and $\MV_2$ represent the same oriented surface-link $F$, then \[\col(\MV_1)=\col(\MV_2)=\col(F).\]
\end{proposition}
\begin{proof}
    Let $\Mcal$ be a movie built from a marked vertex diagram $\MV$ as in the previous paragraph. This movie is composed of consecutive births, followed by a sequence of Reidemeister moves creating the diagram $\MV_A$, then a sequence of saddle movies creating the diagram $\MV_B$, a sequence of Reidemeister moves changing $\MV_B$ into a crossingless diagram, and ending the movie with finitely many deaths. From this description, using Lemma~\ref{rem:movies_of_Reid}, we note that a tribracket coloring for $\Mcal$ is determined by a tribracket coloring of $\MV_A$ and $\MV_B$, which, by definition, are determined by a tribracket coloring of the marked vertex diagram $\MV$. Thus, $\col(\MV)=\col(\Mcal)$. The result follows from Proposition~\ref{prop:invariance_movie_moves} and Lemma~\ref{lem:movies_give_broken}. 
\end{proof}

One can use the movies arising from a marked vertex diagram to relate $\col(F)$ to invariants of marked vertex diagrams $\MV$ such as $|\MV_A|$, $|\MV_B|$, and the number of marked vertices of $\MV$. The proofs of these statements are left as an exercise to the reader. Compare with Lemma~\ref{prop:movie_cross} and Proposition~\ref{prop:movie_inequality_saddle}.

\begin{lemma}\label{prop:marked_versted_ineq_cross}
    Let $\MV$ be a marked vertex diagram representing a surface-link $F\subset S^4$. If $L\subset S^3$ is a link diagram obtained from $\MV$ by replacing each marked vertex with some choice of smoothing, then $$\col(\MV)\leq \col(L).$$
    \end{lemma}

\begin{proposition}\label{prop:marked_vertex_inequality_saddle}
    Let $\MV$ be a marked vertex diagram representing a surface-link $F\subset S^4$. Suppose that $\MV$ has $v$ marked vertices. The following inequalities hold for a finite tribracket $X$. 
    \begin{enumerate}[label=(\roman*)]
    \item $-1+\log_{|X|}\left(\col(F)\right)\leq \min \left(|\MV_A|, |\MV_B|\right)$, and 
    \item $ -2-\chi(F) + 2\log_{|X|}\left(\col(F)\right)\leq v$.
    \end{enumerate}
\end{proposition}

\section{Triplane diagrams}\label{sec:triplane_diagrams}
A \emph{triplane diagram} $\Tcal$ is a triple of $b$-stranded trivial tangles $(T_1,T_2,T_3)$ such that each $T_i$ is a planar diagram for a trivial tangle and the union of any two $T_i\cup \overline{T}_j$ is an unlink diagram. The number of components of $T_1\cup \overline{T}_2$, $T_2\cup \overline{T}_3$, and $T_3\cup \overline{T}_1$ is denoted by $c_1$, $c_2$, and $c_3$, respectively. In this case, we will refer to the tuple $\Tcal=(T_1,T_2,T_3)$ as a $(b;c_1,c_2,c_3)$-triplane diagram. Triplane diagrams were introduced by Meier and Zupan in~\cite{MZ_bridge_trisections_S4}, where they showed that every surface-link in 4-space can be represented by a triplane diagram. They also proved that any two triplane diagrams $\Tcal$ and $\Tcal'$ representing the same surface-link in 4-space are related by a finite sequence of \emph{triplane moves}:
\begin{enumerate}
\item Interior Reidemeister moves: Reidemeister moves on each tangle occurring away from its endpoints. 
\item Mutual braid transpositions: to apply the same braid transposition $\sigma_j^{\pm 1}$ (point swap) between consecutive endpoints to each tangle; see Figure~\ref{fig:braid_move}.
\item Perturbation/deperturbation moves: a local move that increases/decreases the number of strands on each tangle by one. This move is depicted in Figure~\ref{fig:perturbation_move_general}, up to the mirroring and permutation of the tangles. 
\end{enumerate}

\begin{figure}[h]
\centering
\includegraphics[width=.5\textwidth]{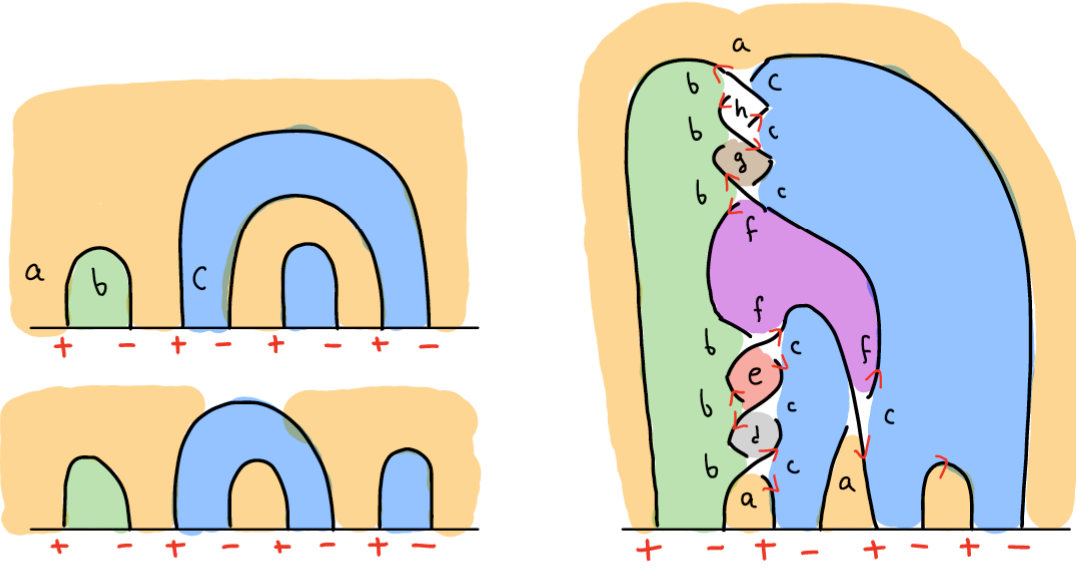}
\caption{Triplane diagram for $9_1$}
\label{fig:9_1}
\end{figure}

\begin{figure}
\centering
\includegraphics[width=.5\textwidth]{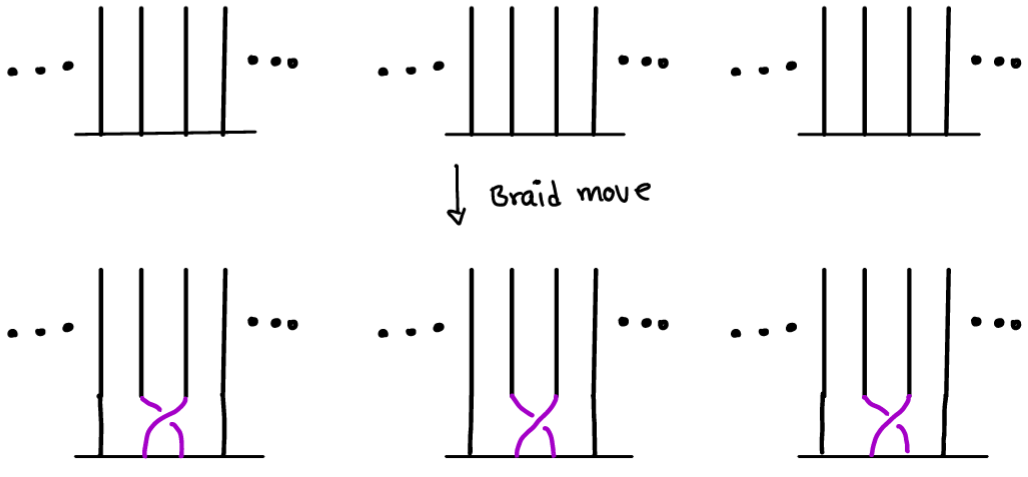}
\caption{Local model of a braid move.}
\label{fig:braid_move}
\end{figure}

An \emph{orientation} of a triplane diagram $\Tcal$ is a choice of orientations for the strands of each tangle of $\Tcal$ such that each endpoint $p\in \{1,\dots, 2b\}$ is a source in each tangle or a sink in each tangle. Equivalently, an orientation of $\Tcal$ is an assignment of $\{+1,-1\}$ to each endpoint of the tangles so that (1) the same endpoints in different tangles have the same label, and (2) every strand of each tangle $T_i$ contains both $\pm 1$. An orientation of $\Tcal$ induces an orientation on each tangle $T_i$ by declaring that the endpoints labeled $+1$ are sinks, and those labeled $-1$ are sources. As shown in Lemma 2.1 of~\cite{MTZ_cubic_graphs}, the orientation of a triplane diagram determines the orientation of the surface-link. Given an oriented triplane diagram $\Tcal=(T_1,T_2,T_3)$, region coloring of the tangles in $\Tcal$ is a \emph{tribracket coloring} if 
\begin{enumerate}
    \item it satisfies the tribracket equations at every crossing of each tangle, and 
    \item near the boundary, the labels between adjacent endpoints are the same in each tangle. 
    \end{enumerate}
For a finite tribracket $X$ and an oriented triplane diagram $\Tcal$, $\col(\Tcal)$ denotes the number of tribracket colorings of $\Tcal$. See Figure~\ref{fig:9_1} for an example of a region coloring that satisfies condition (2); the orientation of the tangles at each crossing forces tribracket relations between the labels; see Table~\ref{table}.

\begin{proposition}\label{prop:invariance_triplane_moves}
Let $\Tcal_1$ and $\Tcal_2$ be two oriented triplane diagrams. If they represent isotopic oriented surface-links, then $\col(\Tcal_1)=\col(\Tcal_2)$.
\end{proposition}
\begin{proof}
It is enough to check that $\col(\Tcal)$ is invariant under triplane moves. The fact that the interior Reidemeister moves do not change the number of tribracket colorings follows from Theorem~\ref{lem:col_IS_knot_diagram}. Suppose that $\Tcal'$ results from performing a mutual braid transposition between adjacent strands as in Figure~\ref{fig:braid_move}. As depicted in the figure, each tangle in $\Tcal'$ has one more region and crossing. Note that the three equations involving the new crossing are identical. Moreover, condition (1) in the definition of a tribracket implies that the color of the new region is determined uniquely by the older three. This induces a bijective correspondence between the tribracket colorings of $\Tcal$ and $\Tcal'$; so $\col(\Tcal)=\col(\Tcal')$. 
\begin{figure}[h]
\centering
\includegraphics[width=.6\textwidth]{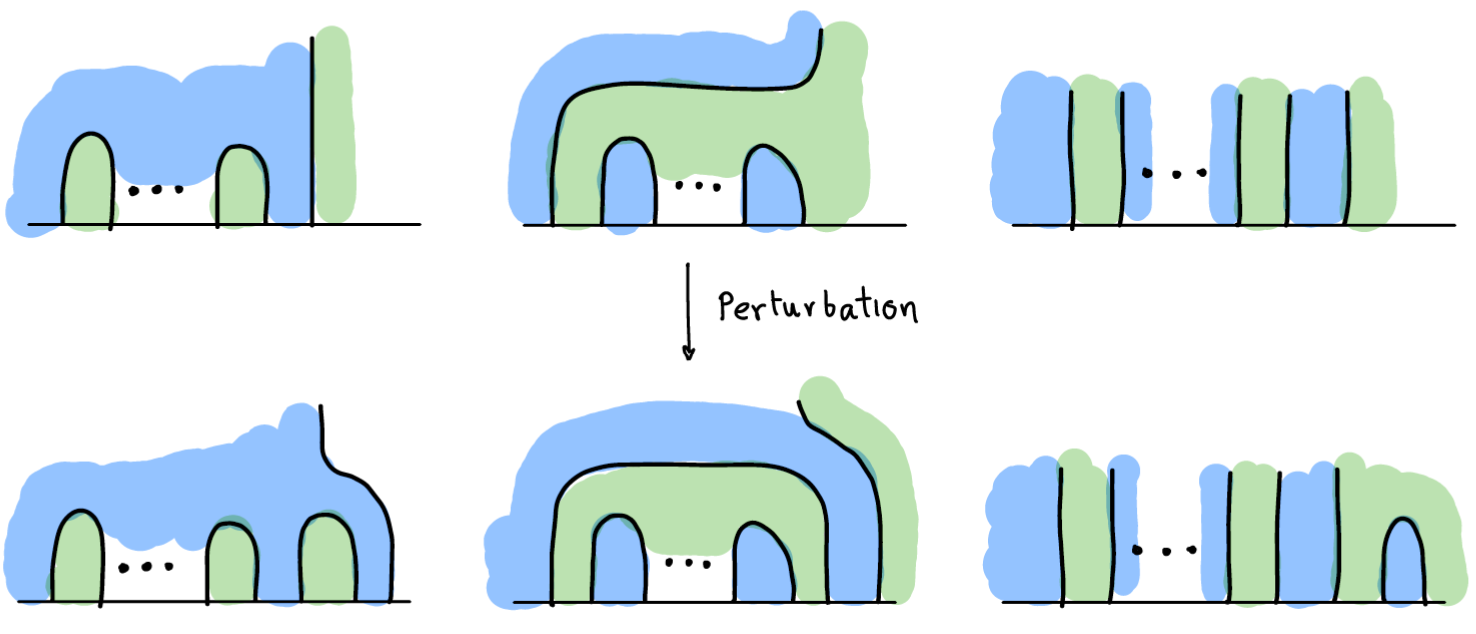}
\caption{Local models for perturbation/deperturbation.}
\label{fig:perturbation_move_general}
\end{figure}

We now discuss the behavior of $\col(\Tcal)$ under perturbation using the local model in Figure~\ref{fig:perturbation_move_general}. Note that there are infinitely many local models for the perturbation depending on the number of strands involved in the model. In any case, the local model forces relations between the colors of the regions involved in the perturbation; see Figure~\ref{fig:perturbation_move_general}. After perturbation, the labels in the new regions extend uniquely using rule (2) of the tribracket coloring of a triplane diagram. Thus $\col(\Tcal)$ is preserved under perturbation. 
\end{proof}

\begin{figure}[h]
\centering
\includegraphics[width=.5\textwidth]{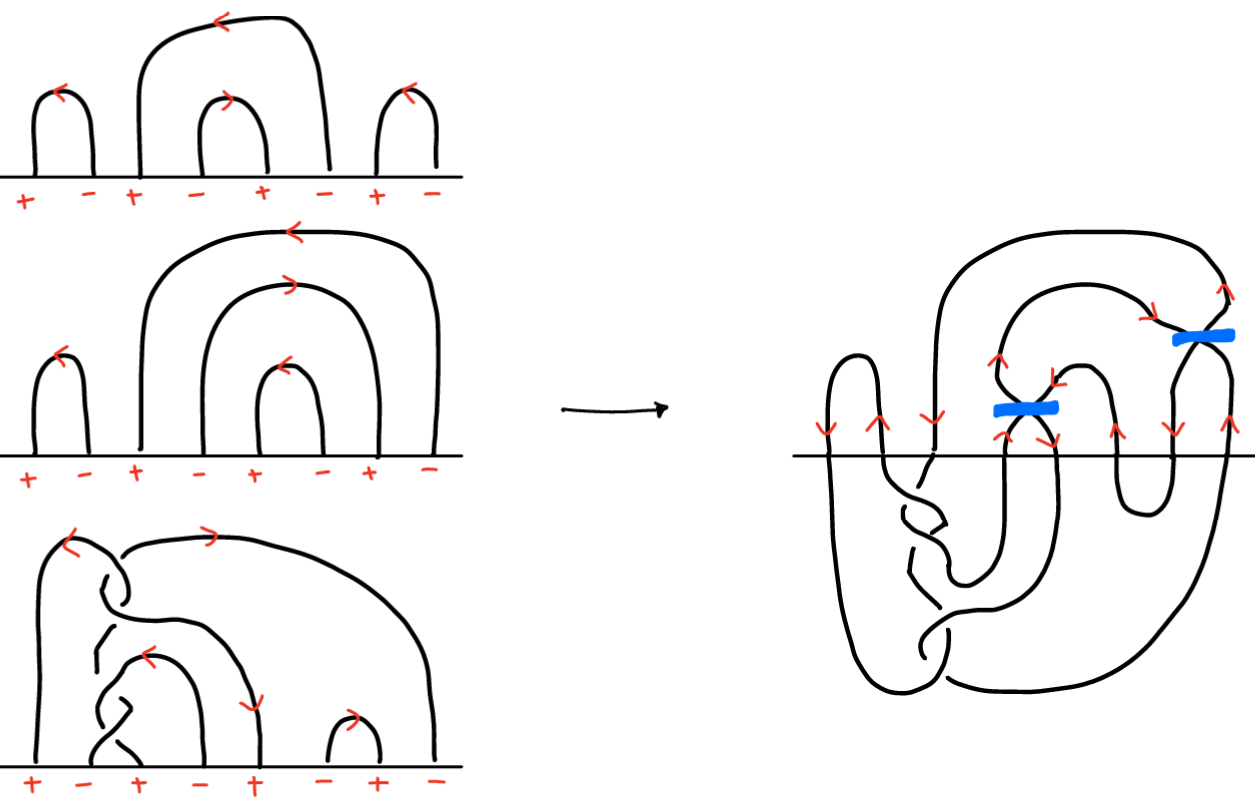}
\caption{A triplane diagram (left) and a marked vertex diagram (right) representing a link of one 2-sphere and one torus. }
\label{fig:triplane_to_marked_2}
\end{figure}
We now explain how to turn a triplane diagram into a marked vertex presentation. This process can be traced back to \cite[Lem 3.3]{MZ_bridge_trisections_S4} to work for arbitrary triplane diagrams. For the sake of simplicity, we will describe a method for a particular set of triplane diagrams called concentrated diagrams. We encourage the curious reader to see how our description is contained in \cite[Sec 4]{Zupan_simple_triplane}.

A triplane diagram is \emph{concentrated} if all its crossings lie in a single tangle. It is known that every triplane diagram can be changed to a concentrated diagram by performing mutual braid moves and interior Reidemeister moves. We are interested in concentrated diagrams with the following property: $T_1$ and $T_2$ are crossingless, and there exists a marked vertex tangle diagram $\MV^{(12)}$ with $(b-c_1)$ marked vertices and no classical crossings such that all-A-smoothing and all-B-smoothing tangle diagrams are equal to the diagrams of $T_1$ and $T_2$, respectively. One way to ensure such a property is to apply further mutual braid moves to our given triplane diagram to turn the tangles $T_1$ and $T_2$ into disjoint unions of the diagrams depicted in Figure~\ref{fig:triplane_to_marked_1}. Some concentrated triplane diagrams already have this property without needing extra braid moves (see Figure~\ref{fig:triplane_to_marked_2}.) Given a concentrated triplane diagram $\Tcal=(T_1,T_2,T_3)$ with all crossings in $T_3$ and a marked vertex diagram $\MV^{(12)}$, we define $\MV_\Tcal=\MV^{(12)}\cup \overline{T}_3$. See Figure~\ref{fig:triplane_to_marked_2} for an example of $\MV_\Tcal$. It follows from~\cite[Lem 3.3]{MZ_bridge_trisections_S4}, that $\MV_\Tcal$ is a marked vertex diagram representing the same surface as $\Tcal$. 

The following result, together with Proposition~\ref{prop:invariance_triplane_moves} imply that $\col(\Tcal)=\col(\MV)$ whenever the triplane $\Tcal$ and the marked vertex diagram $\MV$ represent isotopic oriented surface-links. 
\begin{figure}[h]
\centering
\includegraphics[width=.4\textwidth]{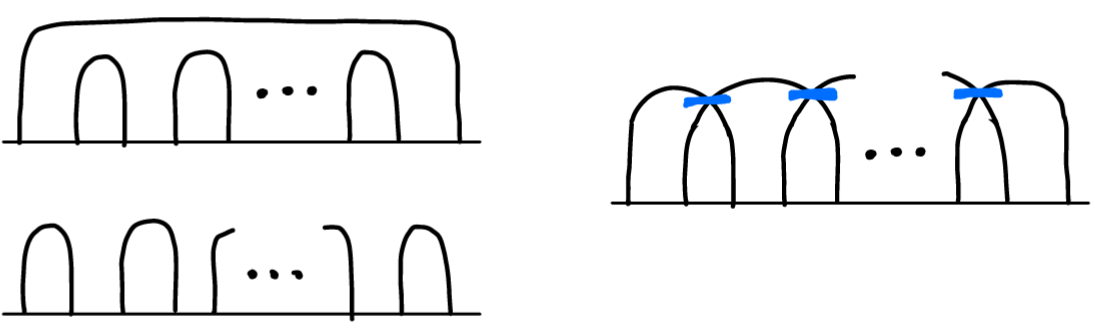}
\caption{Example of markings that take one crossingless tangle to another.}
\label{fig:triplane_to_marked_1}
\end{figure}

\begin{lemma}\label{lem:triplanes_give_movie}
Let $X$ be a tribracket. Let $\Tcal$ be a concentrated triplane diagram and let $\MV_\Tcal=\MV^{(12)}\cup \overline{T}_3$ be a marked vertex diagram described before. There is a bijection between the number of region $X$-colorings of $\Tcal$ and that of $\MV_\Tcal$. In particular, if $X$ is finite, 
\[ \col(\Tcal)=\col(\MV_\Tcal).\] 
\end{lemma}
\begin{proof}
By definition, a tribracket coloring of $\MV_\Tcal$ is determined by tribracket colorings of the all-A-smoothing and all-B-smoothing unlink diagrams. By construction of $\MV_\Tcal$, these two link diagrams equal to $T_1\cup \overline{T}_3$ and $T_2\cup \overline{T}_3$. As all the crossings of $\Tcal$ are concentrated in $T_3$, the set of equations determining a tribracket coloring of $\Tcal$ is the same as for $\MV_\Tcal$. One may worry that the equations for $\MV_\Tcal$ are those of $-\overline{T}_3$. That said, as we see in Figure~\ref{fig:mirror}, they are indeed the same as those of $T_3$. Hence, $\col(\Tcal)=\col(\MV_\Tcal)$. 
\begin{figure}[h]
\centering
\includegraphics[width=.3\textwidth]{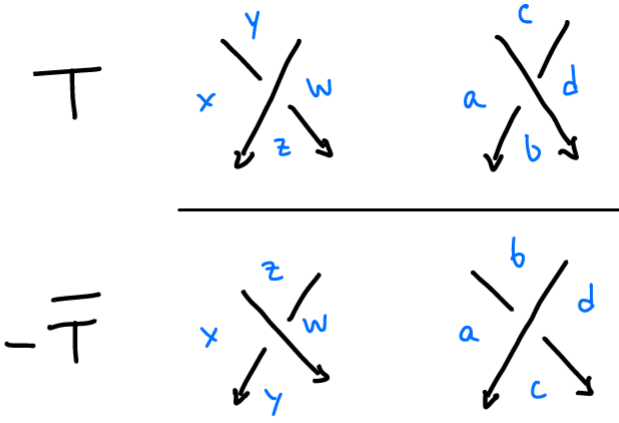}
\caption{The tribracket equations of a tangle and its mirror image with opposite orientation are equal.}
\label{fig:mirror}
\end{figure}
\end{proof}

We now use the region colorings to restrict the patch numbers ($c_i$) and bridge index ($b$). Such quantities can be used to compute the Euler characteristic of the underlying surface. More precisely, if $\Tcal$ is a $(b;c_1,c_2,c_3)$-triplane diagram representing a surface-link $F$, then $\chi(F)=c_1+c_2+c_3-b$. The patch numbers also satisfy that 
\[ \text{rank}\left(\pi_1(S^4-F)\right) \leq \min(c_1,c_2,c_3).\]

\begin{proposition}\label{prop:triplane_inequality_saddle}
    Let $F\subset S^4$ be an oriented surface. Suppose that $F$ admits a $(b;c_1,c_2,c_3)$-triplane diagram. The following inequalities hold for a finite tribracket $X$. 
    \begin{enumerate}[label=(\roman*)]
    \item $-1+\log_{|X|}\left(\col(F)\right)\leq \min(c_1,c_2,c_3)$, and 
    \item $-\chi(F) - 3 +3\log_{|X|}\left(\col(F)\right)\leq b$.
    \end{enumerate}
\end{proposition}
\begin{proof}
Apply mutual braid moves and Reidemeister moves to turn $\Tcal$ into a concentrated diagram $\Tcal'=(T'_1, T'_2, T'_3)$ with all crossings in $T_3$. As $T_1\cup \overline{T}_2$ is a crossingless $c_1$-component unlink, it has $c_1+1$ independent regions and exactly $|X|^{c_1+1}$ tribracket colorings. Each one of such colorings determines the colors near the boundary of the third tangle, which, by Lemma~\ref{lem:colors_of_trivial_tangle}, determines a possible tribracket coloring for $T_3$. In summary, the tribracket colorings of the unlink $T_1\cup \overline{T}_2$  provide a list of all the possible tribracket colorings of $\Tcal$. Hence, $\col(\Tcal)\leq |X|^{c_1+1}$. Repeating the argument with $T_2\cup \overline{T}_3$ and $T_3\cup \overline{T}_1$ gives us the first conclusion. 
The second conclusion follows from the formula relating the Euler characteristic of $F$ and the invariants of $\Tcal$; 
\[ 
-\chi(F) + 3\left(-1+\log_{|X|}\left(\col(F)\right)\right)
\leq 
-\chi(F) + c_1 + c_2 + c_3 
= 
b.
\]
\end{proof}


\subsection{Multiplane diagrams}
After learning that triplane diagrams describe surface-links, one may wonder if longer tuples of trivial tangles also codify surfaces in 4-space. Such a description was first explored by Islambouli, Karimi, Lambert-Cole, and Meier, who used them to study complex curves in $\mathbb{CP}^1\times \mathbb{CP}^1$ \cite{Meier_toric}. We end this section with a quick summary of the theory of region colorings for these objects. The curious reader is referred to Section 3 of~\cite{AE24_Multisections} for more details. 

Let $n \geq 3$, a $(b;c_1,c_2,\dots, c_n)$-\emph{multiplane diagram} (or $n$-plane diagram) is a tuple $\Ical$ of $b$-stranded trivial tangle diagrams $(T_1, T_2, \dots, T_n)$ such that the union of consecutive tangles $T_i\cup \overline{T}_{i+1}$ and $T_n\cup \overline{T}_1$ form unlink diagrams with $c_i$ and $c_n$ components, respectively. Triplanes are examples of multiplane diagrams. Thus, to emphasize that a multiplane diagram may have more than three tangles, we denote multiplane diagrams with $\Ical$ and triplane diagrams with $\Tcal$. Multiplane diagrams of surfaces describe embeddings of surface-links in 4-space. A surface-link may admit multisections with a smaller bridge number ($b$) at the expense of increasing the number of tuples ($n$). For instance, the unknotted torus in 4-space admits a 3-bridge triplane diagram and a 2-bridge 4-plane diagram, both drawn in Figure~\ref{fig:unknotted_torus}. 

\begin{figure}[h]
\centering
\includegraphics[width=.5\textwidth]{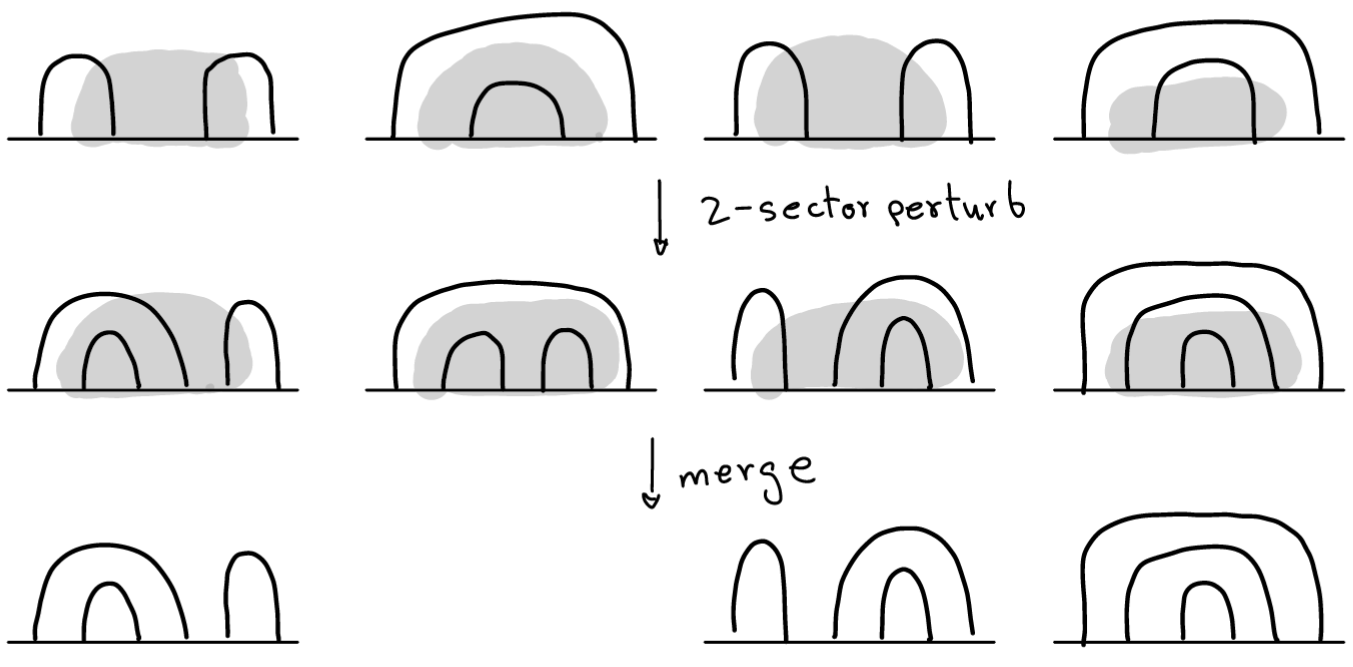}
\caption{Turning a 2-bridge 4-plane diagram of an unknotted torus into a 3-bridge triplane diagram. The shaded regions mark where the local modification from Figure~\ref{fig:two_sector_perturbation} is performed.}
\label{fig:unknotted_torus}
\end{figure}

We can define $\col(\Ical)$ the same way we defined $\col(\Tcal)$; we include it for completeness. A multiplane diagram $\Ical$ is \emph{oriented} if it comes equipped with a consistent assignment of $\{+1,-1\}$ to each puncture so that each strand of each tangle contains both $\pm 1$. For a multiplane diagram $\Ical$, a \emph{tribracket coloring} is a tribracket coloring for each tangle, satisfying that near the boundary of each tangle, the labels between adjacent endpoints are the same on every tangle. For a finite tribracket $X$, $\col(\Ical)$ denotes the number of tribracket colorings of $\Ical$. 

The first author and Engelhardt expanded the set of triplane moves to work for multiplane diagrams~\cite{AE24_Multisections}. They defined two new moves that can turn any multiplane diagram into a triplane diagram of the same surface. The first new move is called the \emph{merge move}, which turns a $(n+1)$-plane diagram into an $n$-plane diagram by removing one redundant tangle from the tuple. For context, we say that a triplane diagram $\Tcal=(T_1,T_2,T_3)$ representing an unlink of $m$ 2-spheres is completely decomposable if it can be obtained by multiple perturbations of the crossingless  $(m;m,m,m)$-triplane diagram. In~\cite[Lem 5.4]{AE24_Multisections}, it was proven that the multiplane diagrams $\Ical=(T_1,T_2,T_3, \dots, T_n)$ and $\Ical'=(T_1,T_3,\dots, T_n)$ represent isotopic surfaces if $\Tcal_0=(T_1,T_2,T_3)$ is a completely decomposed diagram for an unlink of $m$ 2-spheres, where $m=|T_1\cup \overline{T}_3|$. In this case, we say that $\Ical'$ results from a \emph{merge move} on $\Ical$. 

\begin{figure}[h]
\centering
\includegraphics[width=.6\textwidth]{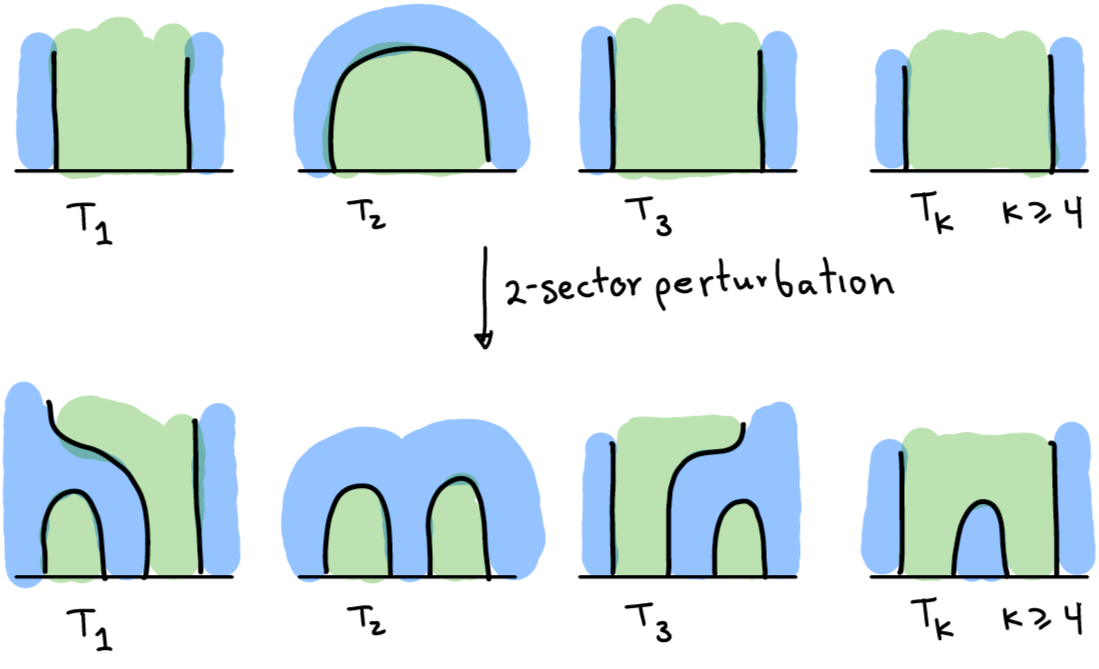}
\caption{The local model of a specific 2-sector perturbation.}
\label{fig:two_sector_perturbation}
\end{figure}
The second move is called a \emph{2-sector perturbation}, which increases the number of strands on each tangle by one. As with triplane diagrams, this is given by a local model that encompasses many kinds of 2-sector perturbations. In Figure~\ref{fig:two_sector_perturbation}, we draw one particular model needed to state the following result.

\begin{lemma}[Prop 5.7 of~\cite{AE24_Multisections}]\label{lem:multi_yield_tris}
Let $\Ical$ be a multiplane diagram. There is a sequence of interior Reidemeister moves, mutual braid moves, merge moves, and 2-sector perturbations (as in Figure~\ref{fig:two_sector_perturbation}) turning $\Ical$ into a triplane diagram $\Tcal$ representing the same surface-link. 
\end{lemma}

We are ready to show that counting tribracket colorings of multiplane diagrams is the same as that of triplane diagrams. Hence, $\col(\Ical)=\col(F)$ is an invariant of the surface-link $F$ represented by $\Ical$. 

\begin{proposition}\label{prop:invariance_multiplane_moves}
If two oriented multiplane diagrams $\Ical_1$ and $\Ical_2$ represent isotopic oriented surface-links, then $\col(\Ical_1)=\col(\Ical_2)$.
\end{proposition}
\begin{proof} 
Let $\Ical$ be an oriented multisection and let $\Tcal$ be the oriented multisection obtained from Lemma~\ref{lem:multi_yield_tris}. To show that $\col(\Ical)=\col(\Tcal)$, it is enough to check the invariance of $\col$ under interior Reidemeister moves, mutual braid moves, merge moves, and 2-sector perturbations as in Figure~\ref{fig:two_sector_perturbation}. The proofs of the first two are the same as in Proposition~\ref{prop:invariance_triplane_moves}. The local model in Figure~\ref{fig:two_sector_perturbation} forces relations between the colors of the regions involved in a 2-sector perturbation; thus, $\col(\Ical)$ does not change under a 2-sector perturbation.

Let $\Ical'$ result from a merge move on $\Ical$. By definition of the merge move, the tuple $\Tcal_0=(T_1,T_2,T_3)$ is a triplane diagram for an unlink of $m=|T_1\cup \overline{T}_3|$ 2-spheres. So Lemmas~\ref{lem:tribrackets_unlink} and~\ref{prop:colorings_unknotted_surface} imply that $\col(\Tcal_0)$ and $\col(T_1\cup \overline{T}_3)$ are equal. Furthermore, there is a bijective correspondence between the tribracket colorings of $(T_1,T_2,T_3)$ and those of $(T_1,T_3)$. In other words, the tribracket equations from $T_2$ are redundant when computing $\col(\Ical)$. Hence, $\col(\Ical)=\col(\Ical')$. 

We have shown that $\col(\Ical)$ and $\col(\Tcal)$ are equal. The proposition follows from the fact that triplane diagrams for isotopic oriented surfaces have the same number of colorings; Proposition~\ref{prop:invariance_triplane_moves}.
\end{proof} 

We end this section by providing bounds between the number of tribracket colorings, the patch numbers ($c_i$), and the bridge index of multiplane diagrams (b). These quantities are related to the underlying surface by the equation $\chi(F)=(2-n)b + \sum_{i=1}^n c_i$. The proof of Lemma~\ref{prop:multisec_ineq_cross} and Proposition~\ref{prop:multis_inequality_bridge} are similar to that of Proposition~\ref{prop:triplane_inequality_saddle}. We leave it as an exercise to the reader.

\begin{lemma}\label{prop:multisec_ineq_cross}
    Let $\Ical$ be a multiplane diagram representing a surface-link $F\subset S^4$. If $L=T_i\cup \overline{T}_j$ is a link diagram obtained by gluing any pair of tangles of $\Ical$, then $$\col(\Ical)\leq \col(L).$$
    \end{lemma}

\begin{proposition}\label{prop:multis_inequality_bridge}
    Let $\Ical$ be a $(b;c_1,\dots, c_n)$-multiplane diagram representing a surface-link $F\subset S^4$. The following inequalities hold for a finite tribracket $X$. 
    \begin{enumerate}[label=(\roman*)]
    \item $-1+\log_{|X|}\left(\col(F)\right)\leq \min \left(c_1, \dots, c_n\right)$, and 
    \item $-n-\chi(F) + n\cdot \log_{|X|}\left(\col(F)\right)\leq (n-2)b$.
    \end{enumerate}
\end{proposition}

\section{Computations}\label{sec:computations}
The following two results are computations of $\col(F)$ in two particular instances: when $X=X_A$ is the Dehn tribracket of an abelian group (Theorem~\ref{thm:abelian_trib}), and when $F$ is a spun knot (Theorem~\ref{thm:spun_knots}). We will use the $\colA(F)$ formulation for triplane diagrams. Note that the tribracket equations for a triplet of tangle diagrams with the same endpoints still make sense regardless of the tangles being trivial or their pairwise union being an unlink. We make this observation because the argument for Theorem~\ref{thm:abelian_trib} below will count the tribracket colorings of some \emph{fake triplane diagrams}. 

Recall that if $A$ is a group, the Dehn tribracket of $A$ is the set $X_A=A$ with the tribracket structure given by $[a,b,c]=ba^{-1}c$. If $A$ is an abelian group, we obtain that $[a,b,c]=-a+(b+c)=[a,c,b]$, and the tribracket equation corresponding to a crossing is the same regardless of which strand goes over/under; see Figure~\ref{fig:tribracket_equations}. Thus, crossing changes do not change the number of $X_A$-tribracket colorings of a triplane diagram. 
Another class of tribrackets that do not see crossings are commutative tribrackets, which satisfy $[a,b,c]=[b,c,a]=[a,c,b]=[c,b,a]$ \cite{Niebrzydowski_tribracket_2}.

\newtheorem*{thm_abelian_trib}{Theorem \ref{thm:abelian_trib}}
\begin{thm_abelian_trib}
Let $A$ be a finite abelian group with associated Dehn tribracket denoted by $X_A$. For any orientable surface-link $F\subset S^4$, the number of $X_A$-colorings is equal to 
\[ 
\text{Col}_{X_A}^R(F)=|A|^{|F|+1}.
\]
\end{thm_abelian_trib}
\begin{proof}
Let $\Tcal$ be an oriented triplane diagram representing $F$. Corollary 1.2 of~\cite{MTZ_cubic_graphs} states that $\Tcal$ can be converted into a tri-plane diagram $\Tcal_0$ for an unknotted surface by a sequence of crossing changes and interior Reidemeister moves. Lemma~\ref{lem:col_IS_knot_diagram} and the assumption that $A$ is abelian imply that all the triplane diagrams in such a sequence have the same number of tribracket colorings. Notice that this is true regardless of whether the sequence taking $\Tcal$ into $\Tcal_0$ passes through some fake triplane diagrams. To end, since $\Tcal_0$ represents an unknotted oriented surface, Proposition~\ref{prop:colorings_unknotted_surface} implies that $\colA(\Tcal_0)=|X_A|^{|F|+1}$. The result follows $$\col(F)=\col(\Tcal)=\colA(\Tcal_0)=|X_A|^{|F|+1}=|A|^{|F|+1}.$$ 
\end{proof}

Artin introduced spun knots as a way to build 4-dimensional knotted spheres from classical knots~\cite{Artin_spinning}. Informally, spinning a knot is the process of building a surface of revolution from a knotted arc, the same way one does in a Calculus class, but one dimension higher. More precisely, consider a knot $K$ in $\R^3$ that lies in the upper-half plane $H^3=\{(x,y,z)| z\geq 0\}$ except for an unknotted arc that dips below the $xy$-plane. The spin of $K$ is the result of rotating the arc $K\cap H^3$ along the $xy$-plane as follows: if $K\cap H^3$ is parametrized by $\left(x(t), y(t), z(t)\right)$, then $\Scal(K)$ is parametrized in $\R^4$ by $\left(x(t), y(t),z(t)\cos{\theta}, z(t)\sin{\theta}\right)$. See \cite[Sec 3]{friedman_spinning} for more details and further generalizations of Artin's spinning construction. Meier and Zupan described a triplane diagram of $\Scal(K)$ that takes as input a diagram of $K$ as in Figure~\ref{fig:spun_knot}~\cite{MZ_bridge_trisections_S4}. 

\begin{figure}[h]
\centering
\includegraphics[width=1\textwidth]{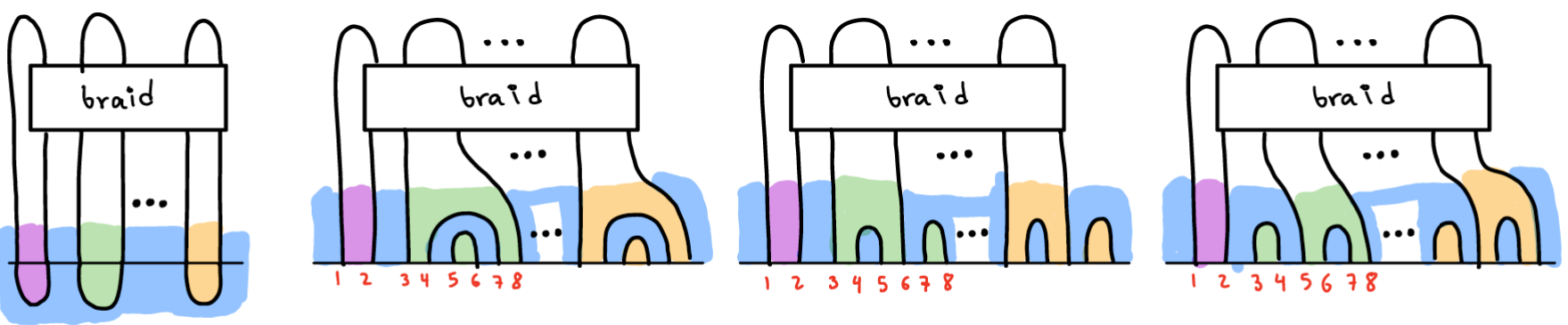}
\caption{(Right) A triplane diagram for $\Scal(K)$ in terms of a plat diagram for $K$ (left).}
\label{fig:spun_knot}
\end{figure}

\newtheorem*{thm_spun_knots}{Theorem \ref{thm:spun_knots}}
\begin{thm_spun_knots}Let $K$ be an oriented knot in $S^3$. There is a bijection between the number of tribracket colorings of $K$ and that of the spun of $K$. In particular, if $X$ is a finite tribracket, then 
\[ 
\col(K) = \col(\Scal(K)).
\]
\end{thm_spun_knots}
\begin{proof}
We pick a knot diagram of $K$ and a triplane diagram $\Tcal$ of $\Scal(K)$ as in Figure~\ref{fig:spun_knot}; note that the braid box in the diagram of $K$ is the same as in each tangle of the triplane diagram of $\Scal(K)$. Furthermore, each local minimum of $K$, except for the left-most one, gets replaced with six consecutive punctures in the tangles of $\Tcal$. For instance, the second local minimum gives rise to punctures 3 through 8. In the figure, we see how a choice of colors near the local minima of $K$ uniquely determines a choice of colors for each tangle of $\Tcal$ satisfying condition (2), and vice-versa. Moreover, as the braid boxes are the same in all tangles of $\Tcal$ and $K$, we note that each coloring of $K$ is a tribracket coloring if and only if the corresponding coloring of $\Tcal$ satisfies the tribracket equations. One can see this as the sub-diagrams of $K$ and $\Tcal$ left to color in Figure~\ref{fig:spun_knot} are all equal. By Lemma~\ref{lem:colors_of_trivial_tangle}, choices of colors near the boundary of each tangle determine unique colorings for each $K$ and $\Tcal$. Thus, $\col(K)=\col(\Scal(K))$. 
\end{proof}

\begin{example}
Table~\ref{table} compiles the coloring equations needed to compute $\col(F)$ for the 2-spheres in Yoshikawa's knotted surface table~\cite{Yoshi94_enumeration}. We found our equations using the triplane diagrams provided in \cite{Zupan_simple_triplane}. One can find the equations of $-F$ quickly from those of $F$ by replacing $[x,y,z]=w$ with $x=[w,y,x]$; see Table~\ref{table}.

Our \texttt{python} computations (in Appendix~\ref{sec:code}) reveal that the 2-knots $9_1$ and $10_3$ are non-invertible; that is, $9_1\not= -9_1$ and $10_3\not=-10_3$. Thus, Theorem~\ref{thm:noninvertible} holds. The tribracket structure that detects this behavior is $X_3=\{1, 2, 3\}$ given by the 3-tensor from \cite[Ex 8]{Nelson_Polynomial}
\[
\begin{bmatrix}
\begin{bmatrix}
1 & 3 & 2 \\
3 & 2 & 1 \\
2 & 1 & 3 
\end{bmatrix}
&
\begin{bmatrix}
3 & 2 & 1 \\
2 & 1 & 2 \\
1 & 3 & 2 
\end{bmatrix}
&
\begin{bmatrix}
2 & 1 & 3 \\
3 & 2 & 1 \\
1 & 3 & 2 
\end{bmatrix}
\end{bmatrix}.
\]
\end{example}

\begin{table}
\caption{2-spheres in Yoshikawa table}
\label{table}
\begin{tabular}{||c | p{2cm} | c | p{8cm} | c ||} 
 \hline
 Surface & Orientation & Variables & Equations & $\text{Col}^{R}_{X_3}(F)$ \\ [0.5ex] 
 \hline\hline
 $0_1$ & Unknot $S^2$ & $a,b$ &   & 9\\ 
 \hline
 $8_1$ & Figure~\ref{fig:trefoil} & $a,\dots, e$ &  $[a,b,c]=e$, $[a,c,d]=e$, $[a,d,b]=e$ & 15\\ 
 \hline
 $-8_1$ & Figure~\ref{fig:trefoil} $\quad$ Reversed & $a,\dots, e$ &  $a=[e,c,b]$, $a=[e,d,c]$, $a=[e,b,d]$ & 15\\ 
 \hline
 $9_1$ & Figure~\ref{fig:9_1} & $a,\dots, f$ & $d=[a,c,b]$, $d=[e,b,c]$, $f=[e,c,b]$, $f=[a,c,c]$, $f=[g,c,b]$, $h=[g,b,c]$, $h=[a,c,b]$ & \textbf{25}\\ 
 \hline
 $-9_1$ & Figure~\ref{fig:9_1} $\quad$ Reversed & $a,\dots, f$ &  $[d,b,c]=a$, $[d,c,b]=e$, $[f,b,c]=e$, $[f,c,c]=a$, $[f,b,c]=g$, $[h,c,b]=g$, $[h,b,c]=a$ & \textbf{21}\\ 
 \hline
 $10_1$ & Figure~\ref{fig:trefoil} & $a,\dots, f$ & $a=[c,b,f]$, $d=[c,f,b]$, $e=[b,a,d]$, $e=[f,d,a]$ & 13\\ 
 \hline
 $-10_1$ & Figure~\ref{fig:trefoil} $\quad$ Reversed & $a,\dots, f$ &  $[a,f,b]=c$, $[d,b,f]=c$, $[e,d,a]=b$, $[e,a,d]=f$ & 13\\ 
 \hline
 $10_2$ & Figure~\ref{fig:10_2} & $a, \dots, j$ & $[e,b,d]=c$, $[a,d,d]=c$, $[a,b,b]=c$, $[e,d,b]=f$, $[a,b,d]=f$, $[i,b,d]=c$, $[g,b,d]=c$, $[i,d,b]=j$, $[a,b,d]=j$, $[g,d,b]=h$, $[a,b,d]=h$ & 37\\ 
 \hline
 $-10_2$ & Figure~\ref{fig:10_2} $\quad$ Reversed & $a,\dots, j$ & $[c,d,b]=e$, $[c,d,d]=a$, $[c,b,b]=a$, $[f,b,d]=e$, $[f,d,b]=a$, $[c,d,b]=i$, $[j,b,d]=i$, $[c,d,b]=g$, $[h,b,d]=g$, $[h,d,b]=a$, $[j,d,b]=a$ & 37\\ 
 \hline
  $10_3$ & Figure~\ref{fig:10_3} & $a, \dots, j$ & $[e,a,c]=b$, $[e,c,a]=d$, $[c,b,d]=f$, $[a,d,b]=f$, $[j,c,a]=b$, $[j,a,c]=d$, $[c,d,b]=i$, $[b,a,i]=d$, $[h,c,a]=b$, $[h,g,c]=b$, $[h,a,g]=b$ & \textbf{14}\\ 
 \hline
 $-10_3$ & Figure~\ref{fig:10_3} $\quad$ Reversed & $a,\dots, j$ & $[b,c,a]=e$, $[d,a,c]=e$, $[f,d,b]=c$, $[f,b,d]=a$, $[b,a,c]=j$, $[d,c,a]=j$, $[i,b,d]=c$, $[d,i,a]=b$, $[b,a,c]=h$, $[b,c,g]=h$, $[b,g,a]=h$ & \textbf{10}\\ 
 \hline
\end{tabular}
\end{table}

\begin{figure}
\centering
\includegraphics[width=.4\textwidth]{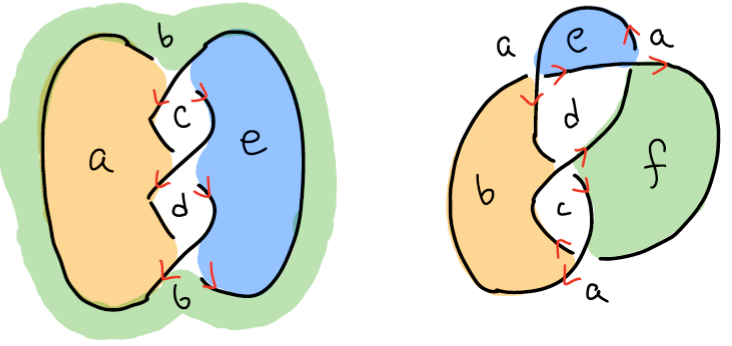}
\caption{$8_1$ and $10_1$ are isotopic to the spin of the trefoil (left) and figure-eight (right) knots, respectively. So their coloring invariants can be computed using Theorem~\ref{thm:spun_knots}.}
\label{fig:trefoil}
\end{figure}

\begin{figure}
\centering
\includegraphics[width=.6\textwidth]{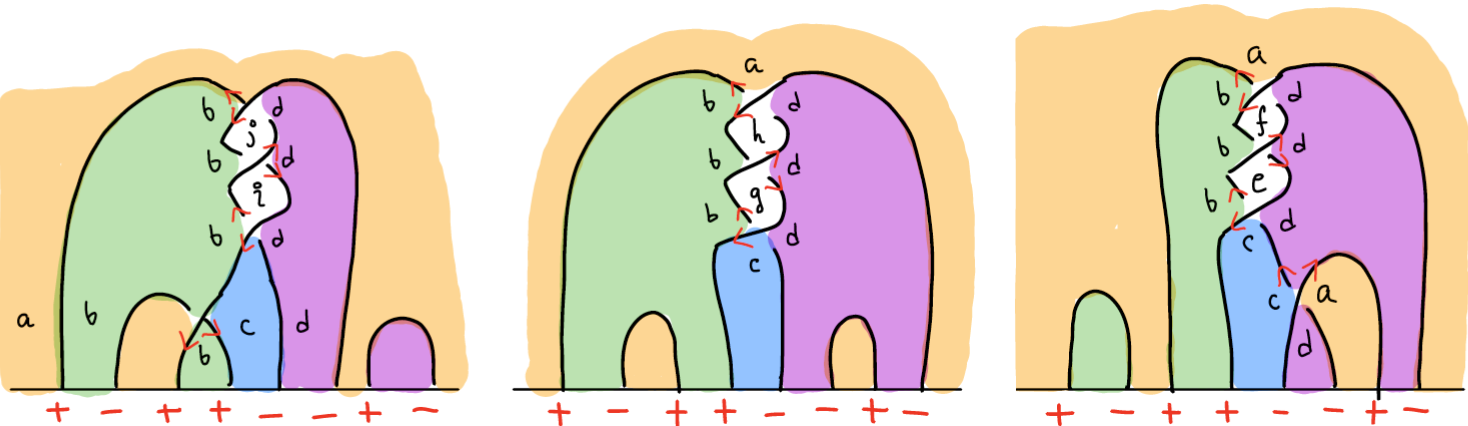}
\caption{$10_2$}
\label{fig:10_2}
\end{figure}

\begin{figure}
\centering
\includegraphics[width=.6\textwidth]{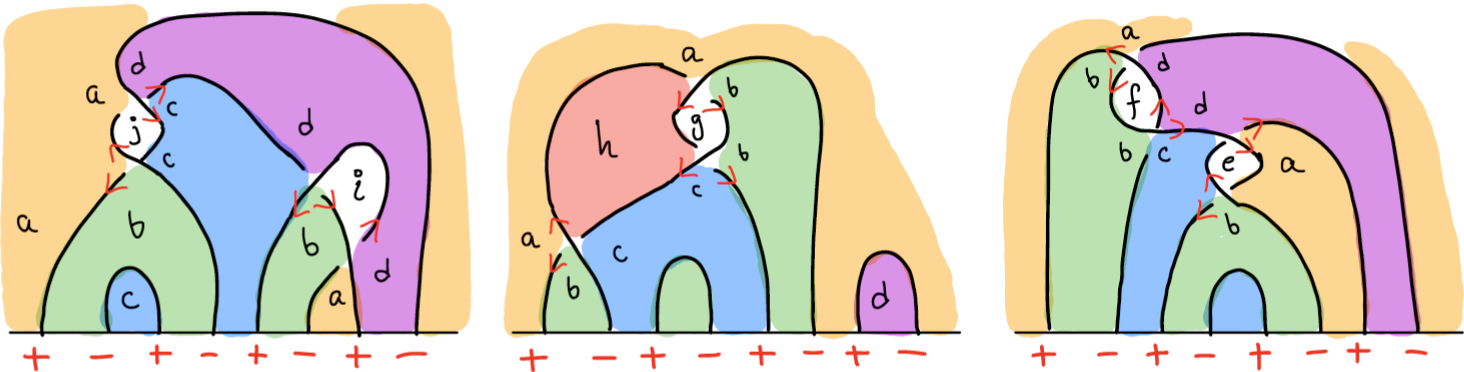}
\caption{$10_3$}
\label{fig:10_3}
\end{figure}

\begin{example}
The surface links $8^{1,1}_1$ and $10^{1,1}_1$ are two-component tori links with the same number of region colorings for any finite tribracket $X$. To see this, consider $8^{1,1}_1$ with the orientation as in Figure~\ref{fig:8_11_1}. The set of region colorings is equal to all the tuples $(a,\dots, f)\in X^6$ satisfying the equations 
$
d=e=f=[a,b,c]=[a,c,b].
$
In particular, $d$, $e$, and $f$, depend uniquely from $(a,b,c)$ and 
\[\col(8^{1,1}_1)=\#\left\{(a,b,c)\in X^3: [a,b,c]=[a,c,b]\right\}.\]
Now consider a different orientation of $8^{1,1}_1$ by changing the blue signs in Figure~\ref{fig:8_11_1}. We obtain a different set of equations on the same tuples $(a,\dots, f)\in X^6$ given by 
\[ 
b=[c,a,f], \text{ } b=[c,f,a], \text{ } b=[c,a,e], \text{ } b=[c,e,a], \text{ } b=[c,a,d], \text{ } b=[c,d,a]. 
\]
By condition (1) of the definition of tribracket, $[c,a,e]=[c,a,f]=[c,a,d]$ implies that $d=e=f$. Thus, $\col(8^{1,1}_1)$ is given by the same set as above $\#\{(a,b,c)\in X^3: [a,b,c]=[a,c,b]\}$. One can check that the same equation holds for the two other orientations of the surface link. 

On the other hand, the link of tori $10^{1,1}_1$ in Figure~\ref{fig:10_11_1} satisfies
\[ 
\col(10^{1,1}_1) = \#\left\{(a,b,c, d)\in X^4: c=[a,b,d]=[a,d,b]\right\} = \#\left\{(a,b,d)\in X^3: [a,b,d]=[a,d,b]\right\}.
\]
The same holds for the other three orientations of $10^{1,1}_1$. 
\end{example}

\begin{figure}[h]
\centering
\includegraphics[width=.6\textwidth]{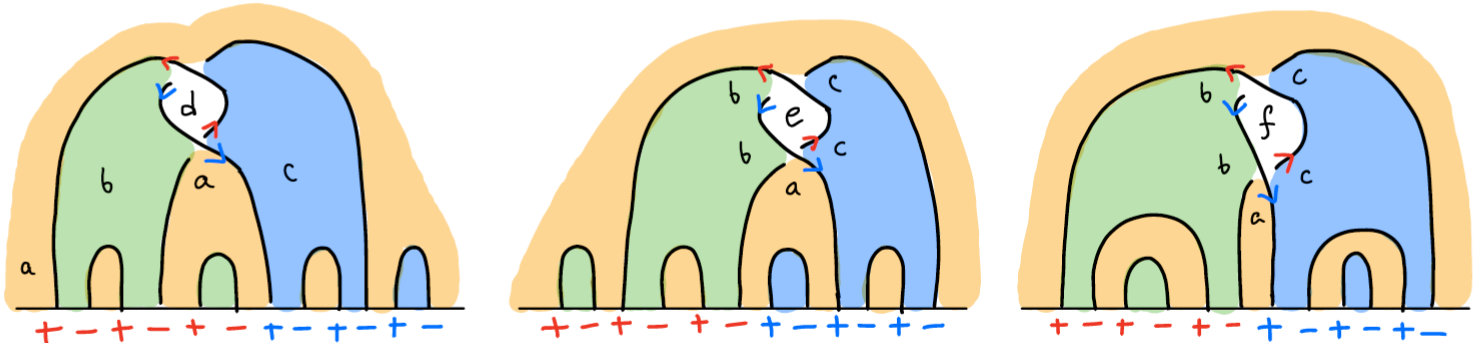}
\caption{$8^{1,1}_1$}
\label{fig:8_11_1}
\end{figure}

\begin{figure}[h]
\centering
\includegraphics[width=.6\textwidth]{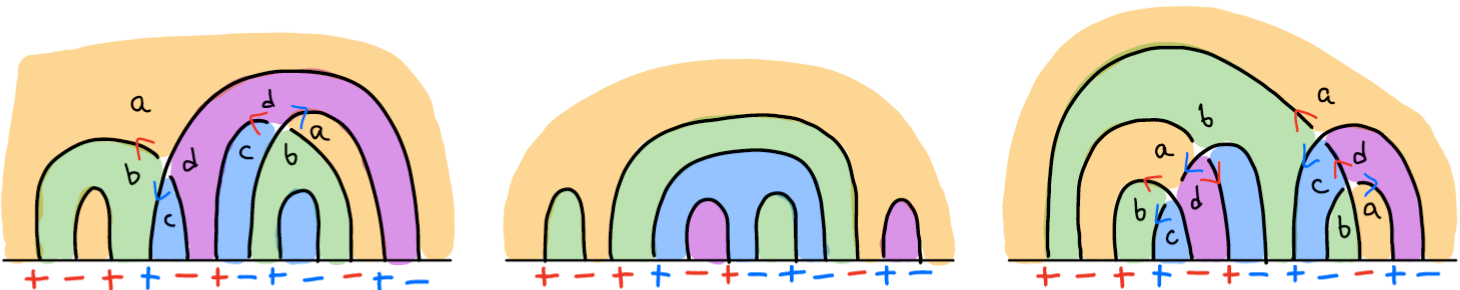}
\caption{$10^{1,1}_1$}
\label{fig:10_11_1}
\end{figure}

\section{Questions}\label{sec:questions}
We include some questions or directions for future works. 
\begin{itemize}
\item Besides abelian tribrackets, for what families of tribrackets are there closed formulae for $\col(F)$? 

\item Are there tribrackets that capture known surface-link invariants? See~\cite{conway_invariants} for a survey.

\item Sato and Tanaka proved a quandle version Proposition~\ref{prop:triplane_inequality_saddle} \cite{kei_trisections}. Is it known which inequality offers more information for the bridge index of $F$?

\item From the work of Sato and Tanaka, and Joseph and Pongtnapaisan, the quandle colorings of even twist spun knots are not that different from those of spun knots~\cite{kei_trisections, Joseph_MRC}. Is there a formula for $\col(\Scal_2m(K))$ for $m\neq 0$ similar to the one in Theorem~\ref{thm:spun_knots}?

\item What is the analog of Propositions~\ref{prop:movie_inequality_saddle} and \ref{prop:triplane_inequality_saddle} for the broken surface diagrams? Is there a relation between $\col(F)$ and the triple point number?

\item How to compute $\col(F)$ using a braid chart?

\item Can tribracket colorings be used to restrict the topology of Seifert solids of 2-knots?

\item How does $\col$ changes under 1-handle addition? 
\end{itemize}
\bibliographystyle{alpha}
\bibliography{sources}

\newcommand{\etalchar}[1]{$^{#1}$}
\begin{thebibliography}{IKLCM24}

\bibitem[AAD{\etalchar{+}}23]{Zupan_simple_triplane}
Wolfgang Allred, Manuel Arag\'on, Zack Dooley, Alexander Goldman, Yucong Lei, Isaiah Martinez, Nicholas Meyer, Devon Peters, Scott Warrander, Ana Wright, and Alexander Zupan.
\newblock Tri-plane diagrams for simple surfaces in {$S^4$}.
\newblock {\em J. Knot Theory Ramifications}, 32(6):Paper No. 2350041, 28, 2023.

\bibitem[AE24]{AE24_Multisections}
Román Aranda and Carolyn Engelhardt.
\newblock Bridge multisections of knotted surfaces in {$S^4$}, 2024.

\bibitem[Art25]{Artin_spinning}
Emil Artin.
\newblock Zur {I}sotopie zweidimensionaler {F}l\"achen im {$R_4$}.
\newblock {\em Abh. Math. Sem. Univ. Hamburg}, 4(1):174--177, 1925.

\bibitem[CKS04]{book_surfaces_in_4space}
Scott Carter, Seiichi Kamada, and Masahico Saito.
\newblock {\em Surfaces in 4-space}, volume 142 of {\em Encyclopaedia of Mathematical Sciences}.
\newblock Springer-Verlag, Berlin, 2004.
\newblock Low-Dimensional Topology, III.

\bibitem[Con22]{conway_invariants}
Anthony Conway.
\newblock Invariants of $2$-knots, 2022.

\bibitem[CS93]{movie_moves}
J.~Scott Carter and Masahico Saito.
\newblock Reidemeister moves for surface isotopies and their interpretation as moves to movies.
\newblock {\em J. Knot Theory Ramifications}, 2(3):251--284, 1993.

\bibitem[CST23]{Trivalent_tribrackets}
Evan Carr, Nancy Scherich, and Sherilyn Tamagawa.
\newblock An invariant of virtual trivalent spatial graphs.
\newblock {\em J. Knot Theory Ramifications}, 32(14):Paper No. 2350087, 15, 2023.

\bibitem[Fri05]{friedman_spinning}
Greg Friedman.
\newblock Knot spinning.
\newblock In {\em Handbook of knot theory}, pages 187--208. Elsevier B. V., Amsterdam, 2005.

\bibitem[GNT18]{Tribracket_Spatial}
Paige Graves, Sam Nelson, and Sherilyn Tamagawa.
\newblock Niebrzydowski algebras and trivalent spatial graphs.
\newblock {\em Internat. J. Math.}, 29(14):1850102, 16, 2018.

\bibitem[IKLCM24]{Meier_toric}
Gabriel Islambouli, Homayun Karimi, Peter Lambert-Cole, and Jeffrey Meier.
\newblock Toric multisections and curves in rational surfaces.
\newblock {\em Indiana Univ. Math. J.}, 73(4):1269--1306, 2024.

\bibitem[JP25]{Joseph_MRC}
Jason Joseph and Puttipong Pongtanapaisan.
\newblock Meridional rank and bridge number of knotted 2-spheres.
\newblock {\em Canad. J. Math.}, 77(1):282--299, 2025.

\bibitem[Liv82]{livingston_unlink}
Charles Livingston.
\newblock Surfaces bounding the unlink.
\newblock {\em Michigan Math. J.}, 29(3):289--298, 1982.

\bibitem[Lom81]{lom81}
S.~J. Lomonaco, Jr.
\newblock The homotopy groups of knots. {I}. {H}ow to compute the algebraic {$2$}-type.
\newblock {\em Pacific J. Math.}, 95(2):349--390, 1981.

\bibitem[MTZ23]{MTZ_cubic_graphs}
Jeffrey Meier, Abigail Thompson, and Alexander Zupan.
\newblock Cubic graphs induced by bridge trisections.
\newblock {\em Math. Res. Lett.}, 30(4):1207--1231, 2023.

\bibitem[MZ17]{MZ_bridge_trisections_S4}
Jeffrey Meier and Alexander Zupan.
\newblock Bridge trisections of knotted surfaces in {$S^4$}.
\newblock {\em Trans. Amer. Math. Soc.}, 369(10):7343--7386, 2017.

\bibitem[Nie14]{Niebrzydowski_tribracket_1}
Maciej Niebrzydowski.
\newblock On some ternary operations in knot theory.
\newblock {\em Fund. Math.}, 225(1):259--276, 2014.

\bibitem[Nie20]{Niebrzydowski_knotted_surfaces}
Maciej Niebrzydowski.
\newblock Homology of ternary algebras yielding invariants of knots and knotted surfaces.
\newblock {\em Algebr. Geom. Topol.}, 20(5):2337--2372, 2020.

\bibitem[NN23]{Nelson_Polynomial}
Sam Nelson and Fletcher Nickerson.
\newblock Polynomial invariants of tribrackets in knot theory.
\newblock {\em Osaka J. Math.}, 60(2):323--332, 2023.

\bibitem[NNS20]{Tribracket_Modules}
Deanna Needell, Sam Nelson, and Yingqi Shi.
\newblock Tribracket modules.
\newblock {\em Internat. J. Math.}, 31(4):2050028, 13, 2020.

\bibitem[NOO19a]{Local_Biquandles}
Sam Nelson, Kanako Oshiro, and Natsumi Oyamaguchi.
\newblock Local biquandles and {N}iebrzydowski's tribracket theory.
\newblock {\em Topology Appl.}, 258:474--512, 2019.

\bibitem[NOO19b]{Nelson_local_biquandles}
Sam Nelson, Kanako Oshiro, and Natsumi Oyamaguchi.
\newblock Local biquandles and {N}iebrzydowski's tribracket theory.
\newblock {\em Topology Appl.}, 258:474--512, 2019.

\bibitem[NP19]{Virtual_Tribrackets}
Sam Nelson and Shane Pico.
\newblock Virtual tribrackets.
\newblock {\em J. Knot Theory Ramifications}, 28(4):1950026, 12, 2019.

\bibitem[NPZD19]{Niebrzydowski_tribracket_2}
Maciej Niebrzydowski, Agata Pilitowska, and Anna Zamojska-Dzienio.
\newblock Knot-theoretic ternary groups.
\newblock {\em Fund. Math.}, 247(3):299--320, 2019.

\bibitem[NPZD20]{Niebrzydowski_Flocks}
Maciej Niebrzydowski, Agata Pilitowska, and Anna Zamojska-Dzienio.
\newblock Knot-theoretic flocks.
\newblock {\em J. Knot Theory Ramifications}, 29(5):2050026, 16, 2020.

\bibitem[Rei74]{Reidemeister_moves}
K.~Reidemeister.
\newblock {\em Knotentheorie}.
\newblock Springer-Verlag, Berlin-New York, 1974.
\newblock Reprint.

\bibitem[Ros98]{Roseman_moves}
Dennis Roseman.
\newblock Reidemeister-type moves for surfaces in four-dimensional space.
\newblock In {\em Knot theory ({W}arsaw, 1995)}, volume~42 of {\em Banach Center Publ.}, pages 347--380. Polish Acad. Sci. Inst. Math., Warsaw, 1998.

\bibitem[ST22]{kei_trisections}
Kouki Sato and Kokoro Tanaka.
\newblock The bridge number of surface links and kei colorings.
\newblock {\em Bull. Lond. Math. Soc.}, 54(5):1763--1771, 2022.

\bibitem[Yos94]{Yoshi94_enumeration}
Katsuyuki Yoshikawa.
\newblock An enumeration of surfaces in four-space.
\newblock {\em Osaka J. Math.}, 31(3):497--522, 1994.

\end{thebibliography}

\appendix
\section{Computing Table~\ref{table}}\label{sec:code}
\noindent To promote the transparency of our computations, we include the functions we used to complete Table~\ref{table}.

\begin{lstlisting}[language=Python,breaklines=true]
from itertools import product

## Equations of Trefoil
def col_trefoil(X,n):
  counteR=0
  for a, b, c, d, e in product(range(n), repeat=5):
    #(a,...,e) runs over [0,n]^5
    if X[a][b][c] % n == e and X[a][c][d] % n == e and X[a][d][b] % n == e :
      counteR+=1
  return(counteR)

#reversed orientation
def col_trefoil_rev(X,n):
  counteR=0
  for a, b, c, d, e in product(range(n), repeat=5):
    if a == X[e][c][b] % n and a == X[e][d][c] % n and a == X[e][b][d] % n:
      counteR+=1
  return(counteR)

#FIGURE-EIGHT EQUATIONS
def col_fig_eight(X,n):
  counteR=0
  for a, b, c, d, e, f in product(range(n), repeat=6):
    if a == X[c][b][f] % n and d == X[c][f][b] % n and e == X[b][a][d] % n and e == X[f][d][a] % n:
      counteR+=1
  return(counteR)

## reversed orientations
def col_fig_eight_rev(X,n):
  counteR=0
  for a, b, c, d, e, f in product(range(n), repeat=6):
    if X[a][f][b] % n == c and X[d][b][f] % n == c and X[e][d][a] % n == b and X[e][a][d] % n == f :
      counteR+=1
  return(counteR)

## 9_1 EQUATIONS
def col_9_1(X,n):
  counteR=0
  for a, b, c, d, e, f, g, h in product(range(n), repeat=8):
    if d == X[a][c][b] % n and d == X[e][b][c] % n and f == X[e][c][b] % n and f == X[a][c][c] % n and f == X[g][c][b] % n and h == X[g][b][c] % n and h == X[a][c][b] % n:
      counteR+=1
  return(counteR)

## reverse orientation
def col_9_1_rev(X,n):
  counteR=0
  for a, b, c, d, e, f, g, h in product(range(n), repeat=8):
    if X[d][b][c] % n == a and X[d][c][b] % n == e and X[f][b][c] % n == e and X[f][c][c] % n == a and X[f][b][c] % n == g and X[h][c][b] % n == g and X[h][b][c] % n == a:
      counteR+=1
  return(counteR)

## 10_2 EQUATIONS
def col_10_2(X,n):
  counteR=0
  for a, b, c, d, e, f, g, h, i, j in product(range(n), repeat=10):
    if X[e][b][d] % n == c and X[a][d][d] % n == c and X[a][b][b] % n == c and X[e][d][b] % n == f and X[a][b][d] % n == f and X[i][b][d] % n == c and X[g][b][d] % n == c and X[i][d][b] % n == j and X[a][b][d] % n == j and X[g][d][b] % n == h and X[a][b][d] % n == h :
      counteR+=1
  return(counteR)

## reverse orientation
def col_10_2_rev(X,n):
  counteR=0
  for a, b, c, d, e, f, g, h, i, j in product(range(n), repeat=10):
    if X[c][d][b] % n == e and X[c][d][d] % n == a and X[c][b][b] % n == a and X[f][b][d] % n == e and X[f][d][b] % n == a and X[c][d][b] % n == i and X[j][b][d] % n == i and X[c][d][b] % n == g and X[h][b][d] % n == g and X[h][d][b] % n == a and X[j][d][b] % n == a :
      counteR+=1
  return(counteR)

## 10_3 EQUATIONS
def col_10_3(X,n):
  counteR=0
  for a, b, c, d, e, f, g, h, i, j in product(range(n), repeat=10):
    if X[e][a][c] % n == b and X[e][c][a] % n == d and X[c][b][d] % n == f and X[a][d][b] % n == f and X[j][c][a] % n == b and X[j][a][c] % n == d and X[c][d][b] % n == i and X[b][a][i] % n == d and X[h][c][a] % n == b and X[h][g][c] % n == b and X[h][a][g] % n == b:
      counteR+=1
  return(counteR)

## reverse orientation
def col_10_3_rev(X,n):
  counteR=0
  for a, b, c, d, e, f, g, h, i, j in product(range(n), repeat=10):
    if X[b][c][a] % n == e and X[d][a][c] % n == e and X[f][d][b] % n == c and X[f][b][d] % n == a and X[b][a][c] % n == j and X[d][c][a] % n == j and X[i][b][d] % n == c and X[d][i][a] % n == b and X[b][a][c] % n == h and X[b][c][g] % n == h and X[b][g][a] % n == h:
      counteR+=1
  return(counteR)

## Tribracket from Example 8 of https://arxiv.org/pdf/2103.02704
Nn=3
Xx=[[
        [1,3,2],
        [3,2,1],
        [2,1,3]
    ],[
        [3,2,1],
        [2,1,3],
        [1,3,2]
    ],[
        [2,1,3],
        [3,2,1],
        [1,3,2]
    ]]
\end{lstlisting}
$\quad $\\
Rom\'an Aranda, {University of Nebraska-Lincoln, Lincoln, NE 68588}\\
Email: \texttt{romanaranda123@gmail.com}\\
URL: \url{https://romanaranda123.wordpress.com/}\\
$\quad$ \\
Noah Crawford, {University of Nebraska-Lincoln, Lincoln, NE 68588}\\
Email: \texttt{ncrawford12@huskers.unl.edu}\\ 
$\quad$ \\
Andrew Maas, {University of Nebraska-Lincoln, Lincoln, NE 68588}\\
Email: \texttt{maas.andrew23@gmail.com}\\ 
$\quad$ \\
Nicole Marienau, {University of Nebraska-Lincoln, Lincoln, NE 68588}\\
Email: \texttt{nmarienau@gmail.com}\\ 
$\quad$ \\
Erica Pearce, {University of Nebraska-Lincoln, Lincoln, NE 68588}\\
Email: \texttt{ericapearce9@gmail.com}\\ 
$\quad$ \\
Renzo Sarreal, {University of Nebraska-Lincoln, Lincoln, NE 68588}\\
Email: \texttt{renzosarreal@gmail.com}\\ 
$\quad$ \\
Savannah Schutte, {University of Nebraska-Lincoln, Lincoln, NE 68588}\\
Email: \texttt{sschutte2004@gmail.com}\\ 
$\quad$ \\
Ransom Sterns, {University of Nebraska-Lincoln, Lincoln, NE 68588}\\
Email: \texttt{ransom.Sterns@gmail.com}\\ 
$\quad$ \\
Eric Woods, {University of Nebraska-Lincoln, Lincoln, NE 68588}\\
Email: \texttt{ewoods3@huskers.unl.edu}\\ 
\end{document}